\documentclass[12pt, reqno]{amsart}
\usepackage{amsmath, amsthm, amscd, amsfonts, amssymb, graphicx, color, mathrsfs}
\usepackage[bookmarksnumbered, colorlinks, plainpages]{hyperref}
\usepackage[all]{xy}
\usepackage{slashed}
\usepackage{tipa}
\usepackage{mathabx}
\usepackage{soul}
\usepackage{cancel}
\usepackage{ulem}
\usepackage{inputenc}
\usepackage{booktabs}

\newtheorem{theorem}{Theorem}[section]
\newtheorem{lemma}[theorem]{Lemma}

\newtheorem{proposition}[theorem]{Proposition}
\newtheorem{corollary}[theorem]{Corollary}
\theoremstyle{definition}
\newtheorem{definition}[theorem]{Definition}
\newtheorem{example}[theorem]{Example}

\theoremstyle{remark}
\newtheorem{remark}[theorem]{Remark}
\numberwithin{equation}{section}

\begin{document}
\setcounter{page}{1}

\title[ Diffusion equations on graded Lie groups]{Analytic functional calculus and G\r{a}rding inequality on graded Lie groups with applications to diffusion equations}

\author[D. Cardona]{Duv\'an Cardona}
\address{
  Duv\'an Cardona:
  \endgraf
  Department of Mathematics: Analysis, Logic and Discrete Mathematics
  \endgraf
  Ghent University, Belgium
  \endgraf
  {\it E-mail address} {\rm duvanc306@gmail.com, duvan.cardonasanchez@ugent.be}
  }
  
  \author[J. Delgado]{Julio Delgado}
\address{
  Julio Delgado:
  \endgraf
  Departmento de Matem\'aticas
  \endgraf
  Universidad del Valle
  \endgraf
  Cali-Colombia
    \endgraf
    {\it E-mail address} {\rm delgado.julio@correounivalle.edu.co}
  }

\author[M. Ruzhansky]{Michael Ruzhansky}
\address{
  Michael Ruzhansky:
  \endgraf
  Department of Mathematics: Analysis, Logic and Discrete Mathematics
  \endgraf
  Ghent University, Belgium
  \endgraf
 and
  \endgraf
  School of Mathematical Sciences
  \endgraf
  Queen Mary University of London
  \endgraf
  United Kingdom
  \endgraf
  {\it E-mail address} {\rm michael.ruzhansky@ugent.be, m.ruzhansky@qmul.ac.uk}
  }

\thanks{The authors are supported  by the FWO  Odysseus  1  grant  G.0H94.18N:  Analysis  and  Partial Differential Equations and by the Methusalem programme of the Ghent University Special Research Fund (BOF)
(Grant number 01M01021). Julio Delgado is also supported by Vice. Inv. Universidad del Valle Grant CI-71281,  MathAmSud and Minciencias-Colombia under the project MATHAMSUD 21-MATH-03.  Michael Ruzhansky is also supported  by EPSRC grant 
EP/R003025/2.
}

\begin{abstract}In this paper we study the Cauchy problem for diffusion equations associated to a class of strongly hypoelliptic pseudo-differential operators on graded Lie groups. To do so, we develop a global complex functional calculus on graded Lie groups in order to analyse the corresponding energy estimates. One of the main aspects of this complex functional calculus is that for the $(\rho,\delta)$-Euclidean H\"ormander classes we recover the standard functional calculus developed by Seeley \cite{seeley}. In consequence the G\r{a}rding inequality that we prove for arbitrary graded Lie groups absorbs the historical 1953's inequality due to G\r{a}rding \cite{Garding1953}.
\end{abstract} \maketitle

\tableofcontents
\allowdisplaybreaks
\section{Introduction}
Let $G$ be a graded Lie group. In this work we study the well-posedness  for the Cauchy problem
\begin{equation}\label{PVI:Intro}(\textnormal{IVP}): \begin{cases}\frac{\partial v}{\partial t}=K(t,x,D)v+f ,& \text{ }v\in \mathscr{D}'((0,T)\times G), 
\\v(0)=u_0, & \text{ } \end{cases}
\end{equation}
associated to a continuous time-dependent family of strongly hypoelliptic pseudo-differential operators $K(t)=K(t,x,D),$ and within the scale of suitable Sobolev spaces \cite{FischerRuzhansky2017}. 

This general framework allows us to consider relevant  examples in the setting of nilpotent Lie groups (e.g. the Heisenberg group and any stratified Lie group). Indeed,  here we consider the situation for a very general class of pseudo-differential operators $K(t)$ in the global  H\"ormander classes on graded Lie groups as developed in \cite{FischerRuzhansky2014',FischerRuzhanskyBook2015}  allowing $K(t) $ to be a hypoelliptic operator of arbitrary order  that cannot be treated with the usual H\"ormander classes defined by localisations (or also with more classical classes).  We recall that several classes of diffusion equations in the temporally homogeneous case can be studied within the setting of the semigroup theory, see e.g. Taylor \cite[Chapter IV]{Taylorbook1981}. 

One of the new aspects in the analysis of \eqref{PVI:Intro} is the construction of  analytic functions of global pseudo-differential operators on graded groups, providing a modern extension of the work of Seeley \cite{seeley}  to the nilpotent setting, and of the global functional calculus on compact Lie groups as developed by the third author and Wirth in \cite{RuzhanskyWirth2014}. 

Graded Lie groups are the family of nilpotent Lie groups characterised by the existence of Rockland operators in view of the Helffer and Nurrigat solution \cite{HelfferNourrigat}  of the Rockland conjecture. For instance, our approach allows us to study (in particular) the Heat equation associated with the following operators:
\begin{itemize}
    \item $K=-\sum_{j}X_j^2,$ that is, the positive sub-Laplacian on a stratified group, 
    \item any positive Rockland operator on a graded Lie group, and any fractional power $\mathcal{R}^{a}$ and $(1+\mathcal{R})^a$ of $\mathcal{R}$ and $1+\mathcal{R},$ respectively, defined by the spectral calculus,
    \item a long variety of  hypoelliptic pseudo-differential operators on arbitrary graded Lie groups (see Theorem \ref{Main:Th} for details).
\end{itemize}
Observe that in the Euclidean case,
the analysis of the regularity for parabolic equations has a long tradition. For instance, a good account on the local structure of solutions and other properties can be found in \cite{edib:book}. On the other hand, fractional diffusion equations have a wide variety of applications e.g. fluid mechanics \cite{Constantin2006,MajdaBertozzi2002},  mathematical finance  \cite{c:ff} and fractional dynamics \cite{skb:kk,Laskin2000,Laskin2012}. The fractional Laplacian and its different generalisations have been of intensive research interest in the last decades, e.g. \cite{rup:f1}, \cite{Cao:hy}, \cite{Uhlm:fl}. 

We observe that in the Euclidean setting in the case of the Laplacian $\mathcal{R}=-\Delta$, that is a Rockland operator of order 2, the heat equation associated to the fractional powers $(-\Delta)^{a}$ and its generalisations are related with diffusion problems \cite{Constantin2006,MajdaBertozzi2002}, and the  relevance of its analysis is justified by its applications in probability, finance, mathematical physics and differential geometry (see  \cite{Grubb2018} and references therein).

It is worth mentioning that the study of the diffusion equation for the operator $(-\Delta)^{a}$  has been undertaken from several perspectives.  Indeed, as it was pointed out in \cite{2021}, roughly speaking: (1) there is a probabilistic approach, based on the fact that it is the infinitesimal generator of a L\'evy process, (2) there is an approach by potential theoretic methods searching for similarities with the usual Laplacian $\Delta$, and (3) there is an approach in terms of  its Fourier representation using that the fractional Laplacian is (modulo an error term) a pseudo-differential operator \cite{HormanderBook34}. The first two methods have filled the major part of contributions in recent years; they deal mainly with real functions and real integral operators (Calder\'on-Zygmund singular integral operators with real kernels). 

The third method gave some early results in the sixties and seventies, indeed, from the pseudo-differential calculus in \cite{KohnNirenberg1965,HormanderBook34} and its generalisations, but has only been reintroduced fairly recently for this particular type of operators. In this work, we will use the pseudo-differential model to study the diffusion equation, e.g. for $\mathcal{R}^{a}$ and $(1+\mathcal{R})^a,$ $a>0.$ Indeed, still the kernels of these pseudo-differential operators become kernels of Calder\'on-Zygmund type \cite{FischerRuzhanskyBook2015}.

Our analysis allows us to tackle fractional powers of the propagator term. Some of the main ideas to study parabolic equations in this work correspond to an adaptation of the methods used for the classical classes of pseudo-differential operators as presented for instance in \cite{Taylorbook1981} and from the recent works \cite{Delgado2016,Delgado2018} in the setting of the Weyl-H\"ormander calculus.

In order to study the well-posedness for the initial value problem (IVP for short) \eqref{PVI:Intro}, it is necessary to extend the functional calculus of Seeley \cite{seeley} in the framework of the global H\"ormander classes of  \cite{FischerRuzhanskyBook2015}. Indeed, to  exploit the  strongly-ellipticity condition on $K(t),$ and to study the corresponding energy estimates we use a G\r{a}rding type inequality \cite{FischerRuzhanskyLowerBounds}, whose proof lies in the construction of the analytic functional calculus.  We want to emphasize that this global functional calculus extends in the context of graded Lie groups the functional calculus on compact Lie groups developed in  \cite{RuzhanskyWirth2014} and \cite[Sections 7 and 8]{CR20} for elliptic and subelliptic H\"ormander classes of operators, respectively. In the case of $G=\mathbb{R}^n,$ and $\mathcal{R}=-\Delta,$ we recover the historical 1953's inequality due to G\r{a}rding \cite{Garding1953}. 

We observe that the inclusion of the spectral calculus of Rockland operators (which is defined by their spectral measures) into the global H\"ormander classes on graded Lie groups was consistently developed in \cite{FischerRuzhanskyBook2015}, so, in this paper we address the problem of the inclusion of analytic functions of general symbols in these classes. As an application of this complex functional calculus of operators we obtain a proof of the G\r{a}rding inequality and we obtain the well-posedness of \eqref{PVI:Intro}. Naturally, the initial data are considered on the Sobolev spaces associated to Rockland operators, as consistently developed in \cite{FischerRuzhansky2017}.

As for versions of the G\r{a}rding inequality and the sharp G\r{a}rding inequality on compact Lie groups with a free of local coordinate systems approach we refer the reader to \cite[Section 8.2]{CR20} and  \cite{CardonaFedericoRuzhansky}. G\r{a}rding type inequalities in the context of the non-harmonic analysis on general smooth manifolds can be found in \cite{CKRTman:1}. 

This paper is organised as follows:
\begin{itemize}
    \item In Section \ref{Preliminaries} we briefly present the basics about the H\"ormander classes on graded groups, for this we follow  \cite{FischerRuzhanskyBook2015}.
    \item In Sections \ref{ParameterGraded} and  \ref{SFC}  we develop the analytic functional calculus  for the global H\"ormander classes on graded groups  \cite{FischerRuzhanskyBook2015}. In constrast with the functional calculus of Seeley \cite{seeley}, where the conditions on the operators are imposed in terms of the principal symbol, here the powers or functions of the operators are included in the pseudo-differential calculus relying on the conditions on the global symbol.
    \item In Subsection \ref{Gardinggraded} we present a proof for the G\r{a}rding inequality on graded Lie groups, improving the one in \cite{FischerRuzhanskyLowerBounds}, where a commutativity condition for the symbol of the operator with the symbol of a Rockland operator is required. Such a improvement is one the strong applications of  the analytic functional calculus that develop here.
    \item Finally, the analysis  for the diffusion problem \eqref{PVI:Intro} is established in Theorem \ref{Main:Th} of Section \ref{GST}. 
\end{itemize}

\section{Preliminaries and global  quantization  on graded Lie groups}\label{Preliminaries}
 The notation and terminology herein on the analysis of homogeneous Lie groups are mostly taken from Folland and Stein \cite{FollandStein1982}. For the theory of pseudo-differential operators we will follow the framework  developed in \cite{FischerRuzhanskyBook2015} through  the notion of (operator-valued) global symbols. If $H_1,H_2$ are Hilbert spaces,  $\mathscr{B}(H_1,H_2)$ denotes the $C^*$-algebra of bounded linear operators from $H_1$ to $H_2,$ and also we will write $\mathscr{B}(H_1)=\mathscr{B}(H_1,H_2).$
    \subsection{Homogeneous and graded Lie groups} 
    Let $G$ be a homogeneous Lie group. This means that $G$ is a connected and simply connected Lie group whose Lie algebra $\mathfrak{g}$ is endowed with a family of dilations $D_{r},$ $r>0,$ which are automorphisms on $\mathfrak{g}$  satisfying the following two conditions:
\begin{itemize}
\item For every $r>0,$ $D_{r}$ is a map of the form
$$ D_{r}=\textnormal{Exp}(\ln(r)A) $$
for some diagonalisable linear operator $A\equiv \textnormal{diag}[\nu_1,\cdots,\nu_n]$ on $\mathfrak{g}.$
\item $\forall X,Y\in \mathfrak{g}, $ and $r>0,$ $[D_{r}X, D_{r}Y]=D_{r}[X,Y].$ 
\end{itemize}
We call  the eigenvalues of $A,$ $\nu_1,\nu_2,\cdots,\nu_n,$ the dilation's weights or weights of $G$.  The homogeneous dimension of a homogeneous Lie group $G$ is given by  $$ Q=\textnormal{\textbf{Tr}}(A)=\sum_{j=1}^n\nu_j.  $$
The family of dilations $D_{r}$ of  $\mathfrak{g}$ induces a family of  maps on $G$ defined via,
$$ D_{r}:=\exp\circ D_{r} \circ \exp^{-1},\,\, r>0, $$
where $\exp:\mathfrak{g}\rightarrow G$ is the  exponential mapping associated to the Lie group $G.$ We refer to the family $D_{r},$ $r>0,$ as dilations on the group. If we write $rx=D_{r}(x),$ $x\in G,$ $r>0,$ then a relation between the homogeneous structure of $G$ and the Haar measure $dx$ on $G$ is given by $$ \int\limits_{G}(f\circ D_{r})(x)dx=r^{-Q}\int\limits_{G}f(x)dx. $$
    
A  Lie group is graded if its Lie algebra $\mathfrak{g}$ may be decomposed as the direct sum of subspaces $\mathfrak{g}=\mathfrak{g}_{1}\oplus\mathfrak{g}_{2}\oplus \cdots \oplus \mathfrak{g}_{s}$ such that $[\mathfrak{g}_{i},\mathfrak{g}_{j} ]\subset \mathfrak{g}_{i+j},$ and $ \mathfrak{g}_{i+j}=\{0\}$ if $i+j>s.$  Examples of such groups are the Heisenberg group $\mathbb{H}^n$ and more generally any stratified groups where the Lie algebra $ \mathfrak{g}$ is generated by $\mathfrak{g}_{1}$.  Here, $n$ is the topological dimension of $G,$ $n=n_{1}+\cdots +n_{s},$ where $n_{k}=\mbox{dim}\mathfrak{g}_{k}.$

A Lie algebra admitting a family of dilations is nilpotent, and hence so is its associated
connected, simply connected Lie group. The converse does not hold, i.e., not every
nilpotent Lie group is homogeneous (see Dyer \cite{Dyer1970}) although they exhaust a large class (see Johnson \cite[page 294]{Johnson1975}). Indeed, the main class of Lie groups under our consideration is that of graded Lie groups. A graded Lie group $G$ is a homogeneous Lie group equipped with a family of weights $\nu_j,$ all of them positive rational numbers. Let us observe that if $\nu_{i}=\frac{a_i}{b_i}$ with $a_i,b_i$ integer numbers,  and $b$ is the least common multiple of the $b_i's,$ the family of dilations 
$$ \mathbb{D}_{r}=\textnormal{Exp}(\ln(r^b)A):\mathfrak{g}\rightarrow\mathfrak{g}, $$
has integer weights,  $\nu_{i}=\frac{a_i b}{b_i}. $ So, in this paper we always assume that the weights $\nu_j,$ defining the family of dilations are non-negative integer numbers which allow us to assume that the homogeneous dimension $Q$ is a non-negative integer number. This is a natural context for the study of Rockland operators (see Remark 4.1.4 of \cite{FischerRuzhanskyBook2015}).

\subsection{Fourier analysis on nilpotent Lie groups}
Let $G$ be a simply connected nilpotent Lie group.  
Let us assume that $\pi$ is a continuous, unitary and irreducible  representation of $G,$ this means that,
\begin{itemize}
    \item $\pi\in \textnormal{Hom}(G, \textnormal{U}(H_{\pi})),$ for some separable Hilbert space $H_\pi,$ i.e. $\pi(gg')=\pi(g)\pi(g')$ and for the  adjoint of $\pi(g),$ $\pi(g)^*=\pi(g^{-1}),$ for every $g,g'\in G.$
    \item The map $(g,v)\mapsto \pi(g)v, $ from $G\times H_\pi$ into $H_\pi$ is continuous.
    \item For every $g\in G,$ and $W_\pi\subset H_\pi,$ if $\pi(g)W_{\pi}\subset W_{\pi},$ then $W_\pi=H_\pi$ or $W_\pi=\emptyset.$
\end{itemize} Let $\textnormal{Rep}(G)$ be the family of unitary, continuous and irreducible representations of $G.$ The relation, {\small{
\begin{equation*}
    \pi_1\sim \pi_2\textnormal{ if and only if, there exists } A\in \mathscr{B}(H_{\pi_1},H_{\pi_2}),\textnormal{ such that }A\pi_{1}(g)A^{-1}=\pi_2(g), 
\end{equation*}}}for every $g\in G,$ is an equivalence relation and the unitary dual of $G,$ denoted by $\widehat{G}$ is defined via
$
    \widehat{G}:={\textnormal{Rep}(G)}/{\sim}.
$ Let us denote by $d\pi$ the Plancherel measure on $\widehat{G}.$ 
The Fourier transform of $f\in \mathscr{S}(G), $ (this means that $f\circ \textnormal{exp}\in \mathscr{S}(\mathfrak{g})$, with $\mathfrak{g}\simeq \mathbb{R}^{\dim(G)}$) at $\pi\in\widehat{G},$ is defined by 
\begin{equation*}
    \widehat{f}(\pi)=\int\limits_{G}f(g)\pi(g)^*dg:H_\pi\rightarrow H_\pi,\textnormal{   and   }\mathscr{F}_{G}:\mathscr{S}(G)\rightarrow \mathscr{S}(\widehat{G}):=\mathscr{F}_{G}(\mathscr{S}(G)).
\end{equation*}
If we identify one representation $\pi$ with its equivalence class, $[\pi]=\{\pi':\pi\sim \pi'\}$,  for every $\pi\in \widehat{G}, $ the Kirillov trace character $\Theta_\pi$ defined by  $$(\Theta_{\pi},f):
=\textnormal{\textbf{Tr}}(\widehat{f}(\pi)),$$ is a tempered distribution on $\mathscr{S}(G).$ In particular, the identity
$
    f(e_G)=\int\limits_{\widehat{G}}\textnormal{\textbf{Tr}}(\widehat{f}(\pi))d\pi,
$ 
implies the Fourier inversion formula $f=\mathscr{F}_G^{-1}(\widehat{f}),$ where
\begin{equation*}
    (\mathscr{F}_G^{-1}\sigma)(g):=\int\limits_{\widehat{G}}\textnormal{\textbf{Tr}}[\pi(g)\sigma(\pi)]d\pi,\,\,g\in G,\,\,\,\,\mathscr{F}_G^{-1}:\mathscr{S}(\widehat{G})\rightarrow\mathscr{S}(G),
\end{equation*}is the inverse Fourier  transform. In this context, the Plancherel theorem takes the form $\Vert f\Vert_{L^2(G)}=\Vert \widehat{f}\Vert_{L^2(\widehat{G})}$,  where  $ L^2(\widehat{G}):=\int\limits_{\widehat{G}}H_\pi\otimes H_{\pi}^*d\pi,$  is the Hilbert space endowed with the norm: $\Vert \sigma\Vert_{L^2(\widehat{G})}=(\int_{\widehat{G}}\Vert \sigma(\pi)\Vert_{\textnormal{HS}}^2d\pi)^{\frac{1}{2}}.$

\subsection{Homogeneous linear operators and Rockland operators} A linear operator $T:C^\infty(G)\rightarrow \mathscr{D}'(G)$ is homogeneous of  degree $\nu\in \mathbb{C}$ if for every $r>0$ the equality 
\begin{equation*}
T(f\circ D_{r})=r^{\nu}(Tf)\circ D_{r}
\end{equation*}
holds for every $f\in \mathscr{D}(G). $
If for every representation $\pi\in\widehat{G},$ $\pi:G\rightarrow U({H}_{\pi}),$ we denote by ${H}_{\pi}^{\infty}$ the set of smooth vectors, that is, the space of elements $v\in {H}_{\pi}$ such that the function $x\mapsto \pi(x)v,$ $x\in \widehat{G}$ is smooth,  a Rockland operator is a left-invariant differential operator $\mathcal{R}$ which is homogeneous of positive degree $\nu=\nu_{\mathcal{R}}$ and such that, for every unitary irreducible non-trivial representation $\pi\in \widehat{G},$ $\pi({R})$ is injective on ${H}_{\pi}^{\infty};$ $\sigma_{\mathcal{R}}(\pi)=\pi(\mathcal{R})$ is the symbol associated to $\mathcal{R}.$ It coincides with the infinitesimal representation of $\mathcal{R}$ as an element of the universal enveloping algebra. It can be shown that a Lie group $G$ is graded if and only if there exists a differential Rockland operator on $G.$ If the Rockland operator is formally self-adjoint, then $\mathcal{R}$ and $\pi(\mathcal{R})$ admit self-adjoint extensions on $L^{2}(G)$ and ${H}_{\pi},$ respectively. Now if we preserve the same notation for their self-adjoint
extensions and we denote by $E$ and $E_{\pi}$  their respective spectral measures, we will denote by
$$ f(\mathcal{R})=\int\limits_{-\infty}^{\infty}f(\lambda) dE(\lambda),\,\,\,\textnormal{and}\,\,\,\pi(f(\mathcal{R}))\equiv f(\pi(\mathcal{R}))=\int\limits_{-\infty}^{\infty}f(\lambda) dE_{\pi}(\lambda), $$ the functions defined by the functional calculus. 
In general, we will reserve the notation ${E_A(\lambda)}_{0\leq\lambda<\infty}$ for the spectral measure associated with a positive and self-adjoint operator $A$ on a Hilbert space $H.$

 {{For  any $\alpha\in \mathbb{N}_0^n,$ and for an arbitrary family $X_1,\cdots, X_n,$ of left-invariant  vector-fields we will use the notation
\begin{equation}
    [\alpha]:=\sum_{j=1}^n\nu_j\alpha_j,
\end{equation}for the homogeneity degree of the operator $X^{\alpha}:=X_1^{\alpha_1}\cdots X_{n}^{\alpha_n},$ whose order is $|\alpha|:=\sum_{j=1}^n\alpha_j.$}}

\subsection{Symbols and quantization of pseudo-differential operators}
 In order to present a consistent definition of pseudo-differential operators as developed in \cite{FischerRuzhanskyBook2015} (see the quantisation formula  \eqref{Quantization}), we need  a suitable class of spaces on the unitary dual $\widehat{G}$ acting in a suitable way with the set of smooth vectors $H_{\pi}^{\infty},$ on every representation space $H_{\pi}.$ Let us now recall the main notions.
 \begin{definition}[Sobolev spaces on  smooth vectors] Let $\pi_1\in \textnormal{Rep}(G),$ and $a\in \mathbb{R}.$ We denote by $H_{\pi_1}^a,$ the Hilbert space obtained  as the completion of $H_{\pi_1}^\infty$ with respect to the norm
 \begin{equation*}
     \Vert v \Vert_{H_{\pi_1}^a}=\Vert\pi_1(1+\mathcal{R})^{\frac{a}{\nu}} v\Vert_{H_{\pi_1}},
 \end{equation*}
 where $\mathcal{R}$ is  a positive Rockland operator on $G$ of homogeneous degree $\nu>0.$
 \end{definition}
 
 In order to introduce the general notion of symbol as the one developed in \cite{FischerRuzhanskyBook2015}, we will use a suitable notion of operator-valued  symbols acting on smooth vectors. We introduce it as follows.
 
\begin{definition}
A $\widehat{G}$-field of operators $\sigma=\{\sigma(\pi):\pi\in \widehat{G}\}$ defined on smooth vectors {is defined} on the Sobolev space ${H}_\pi^a$ when for each representation $\pi_1\in \textnormal{Rep}(G),$ the operator $\sigma(\pi_1)$ is bounded from $H^a_{\pi_1}$ into $H_{\pi_{1}}$ in the sense that
\begin{equation*}
    \sup_{ \Vert v\Vert_{H_{\pi_1}^a}=1}\Vert \sigma(\pi_1)v \Vert<\infty.
\end{equation*}
\end{definition}

We will consider those $\widehat{G}$-fields of operators with ranges in Sobolev spaces on smooth vectors.  We recall that the  Sobolev space $L^{2}_{a}(G)$ is  defined by the norm (see \cite[Chapter 4]{FischerRuzhanskyBook2015})
\begin{equation}\label{L2ab2}
    \Vert f \Vert_{L^{2}_{a}(G)}=\Vert (1+\mathcal{R})^{\frac{a}{\nu}}f\Vert_{L^2(G)},
\end{equation} for $a\in \mathbb{R}.$

\begin{definition}
A $\widehat{G}$-field of operators  defined on smooth vectors {with range} in the Sobolev space $H_{\pi}^a$ is a family of classes of operators $\sigma=\{\sigma(\pi):\pi\in \widehat{G}\}$ where
\begin{equation*}
    \sigma(\pi):=\{\sigma(\pi_1):H^{\infty}_{\pi_1}\rightarrow H_{\pi}^a,\,\,\pi_1\in \pi\},
\end{equation*} for every $\pi\in \widehat{G}$ viewed as a subset of $\textnormal{Rep}(G),$ satisfying for every two elements $\sigma(\pi_1)$ and $\sigma(\pi_2)$ in $\sigma(\pi):$
\begin{equation*}
  \textnormal{If  }  \pi_{1}\sim \pi_2  \textnormal{  then  }   \sigma(\pi_1)\sim \sigma(\pi_2). 
\end{equation*}
\end{definition}
  The following notion  will be useful in order to use the general theory of non-commutative integration (see e.g. Dixmier \cite{Dixmier1953}). 
\begin{definition}
A $\widehat{G}$-field of operators  defined on smooth vectors with range in the Sobolev space $H_\pi^a$ {is measurable}  when for some (and hence for any) $\pi_1\in \pi$ and any vector $v_{\pi_1}\in H_{\pi_1}^\infty,$ as $\pi\in \widehat{G},$ the resulting field $\{\sigma(\pi)v_\pi:\pi\in\widehat{G}\},$ 
is $d\pi$-measurable and
\begin{equation*}
    \int\limits_{\widehat{G} }\Vert v_\pi \Vert^2_{H_\pi^a}d\pi=\int\limits_{\widehat{G} }\Vert\pi(1+\mathcal{R})^{\frac{a}{\nu}} v_\pi \Vert^2_{H_\pi}d\pi<\infty.
\end{equation*}
\end{definition}
\begin{remark}
We always assume that a $\widehat{G}$-field of operators  defined on smooth vectors {with range} in the Sobolev space $H_{\pi}^a$ is $d\pi$-measurable.
\end{remark} The $\widehat{G}$-fields of operators associated to Rockland operators can be defined as follows.
\begin{definition}
Let $L^2_a(\widehat{G})$ denote the space of fields of operators $\sigma$ with range in $H_\pi^a,$ that is,
\begin{equation*}
    \sigma=\{\sigma(\pi):H_\pi^\infty\rightarrow H_\pi^a\}, \textnormal{ with }\{\pi(1+\mathcal{R})^{\frac{a}{\nu}}\sigma(\pi):\pi\in \widehat{G}\}\in L^2(\widehat{G}),
\end{equation*}for one (and hence for any) Rockland operator of homogeneous degree $\nu.$ We also denote
\begin{equation*}
    \Vert \sigma\Vert_{L^2_a(\widehat{G})}:=\Vert \pi(1+\mathcal{R})^{\frac{a}{\nu}}\sigma(\pi)\Vert_{L^2(\widehat{G})}.
\end{equation*}
\end{definition} By using the notation above, we will introduce a family of function spaces required to define $\widehat{G}$-fields of operators (that will be used to define the symbol of a pseudo-differential operator, see Definition \ref{SRCK}).
\begin{definition}[The spaces $\mathscr{L}_{L}(L^2_a(G),L^2_b(G)),$ $\mathcal{K}_{a,b}(G)$ and $L^\infty_{a,b}(\widehat{G})$]
\hspace{0.1cm}
\begin{itemize}
    \item The space $\mathscr{L}_{L}(L^2_a(G),L^2_b(G)) $ consists of all  left-invariant linear operators $T$  such that  $T:L^2_a(G)\rightarrow L^2_b(G) $ extends to a bounded operator.
    \item The space  $\mathcal{K}_{a,b}(G)$ is the family of all right convolution kernels of elements in $  \mathscr{L}_{L}(L^2_a(G),L^2_b(G))  ,$ i.e. $k=T\delta\in \mathcal{K}_{a,b}(G)$ if and only if  $T\in
    \mathscr{L}_{L}(L^2_a(G),L^2_b(G)) .$ 
    \item We also define the space $L^\infty_{a,b}(\widehat{G})$ by the following condition:  $\sigma\in L^\infty_{a,b}(\widehat{G})$ if 
\begin{equation*}
    \Vert \pi(1+\mathcal{R})^{\frac{b}{\nu}}\sigma(\pi)\pi(1+\mathcal{R})^{-\frac{a}{\nu}} \Vert_{L^\infty(\widehat{G})}:=\sup_{\pi\in\widehat{G}}\Vert \pi(1+\mathcal{R})^{\frac{b}{\nu}}\sigma(\pi)\pi(1+\mathcal{R})^{-\frac{a}{\nu}} \Vert_{\mathscr{B}(H_\pi)}<\infty.
\end{equation*}
\end{itemize} In this case $T_\sigma:L^2_a(G)\rightarrow L^2_b(G)$ extends to a bounded operator with 
\begin{equation*}
  \Vert \sigma\Vert _{L^\infty_{a,b}(\widehat{G})}= \Vert T_{\sigma} \Vert_{\mathscr{L}(L^2_a(G),L^2_b(G))},
\end{equation*} and   $\sigma\in L^\infty_{a,b}(\widehat{G})$ if and only if $k:=\mathscr{F}_{G}^{-1}\sigma \in \mathcal{K}_{a,b}(G).$
\end{definition}
   With the previous definitions, we will introduce the type of symbols that we will use  in this work and under which the quantization formula makes sense.
   
   \begin{definition}[Symbols and right-convolution kernels]\label{SRCK} A {symbol} is a field of operators $\{\sigma(x,\pi):H_\pi^\infty\rightarrow H_\pi^\infty,\,\,\pi\in\widehat{G}\},$ depending on $x\in G,$ such that 
   \begin{equation*}
       \sigma(x,\cdot)=\{\sigma(x,\pi):H_\pi^\infty\rightarrow H_\pi^\infty,\,\,\pi\in\widehat{G}\}\in L^\infty_{a,b}(\widehat{G})
   \end{equation*}for some $a,b\in \mathbb{R}.$ The {right-convolution kernel} $k\in C^\infty(G,\mathscr{S}'(G))$ associated with $\sigma$ is defined, via the inverse Fourier transform on the group by
   \begin{equation*}
       x\mapsto k(x)\equiv k_{x}:=\mathscr{F}_{G}^{-1}(\sigma(x,\cdot)): G\rightarrow\mathscr{S}'(G).
   \end{equation*}
   \end{definition}
   Definition \ref{SRCK} in this section allows us to establish the following theorem, which gives sense to the quantization of pseudo-differential operators in the graded setting (see Theorem 5.1.39 of \cite{FischerRuzhanskyBook2015}).
   \begin{theorem}\label{thetheoremofsymbol}
   Let us consider a symbol $\sigma$ and its associated right-convolution  kernel $k.$ For every $f\in \mathscr{S}(G),$ let us define the operator $A$ acting on $\mathscr{S}(G),$ via 
   \begin{equation}\label{pseudo}
       Af(x)=(f\ast k_{x})(x),\,\,\,\,\,x\in G.
   \end{equation}Then $Af\in C^\infty,$ and 
   \begin{equation}\label{Quantization}
       Af(x)=\int\limits_{\widehat{G}}\textnormal{\textbf{Tr}}(\pi(x)\sigma(x,\pi)\widehat{f}(\pi))d\pi.
   \end{equation}
   \end{theorem}
   
 Theorem   \ref{thetheoremofsymbol} motivates the following definition.
 \begin{definition}\label{DefiPSDO}
A continuous linear operator $A:C^\infty(G)\rightarrow\mathscr{D}'(G)$ with Schwartz kernel $K_A\in C^{\infty}(G)\widehat{\otimes}_{\pi} \mathscr{D}'(G),$ is a pseudo-differential operator, if
 there exists a  \textit{symbol}, which is a field of operators $\{\sigma(x,\pi):H_\pi^\infty\rightarrow H_\pi^\infty,\,\,\pi\in\widehat{G}\},$ depending on $x\in G,$ such that 
   \begin{equation*}
       \sigma(x,\cdot)=\{\sigma(x,\pi):H_\pi^\infty\rightarrow H_\pi^\infty,\,\,\pi\in\widehat{G}\}\in L^\infty_{a,b}(\widehat{G})
   \end{equation*}for some $a,b\in \mathbb{R},$ such that, the Schwartz kernel of $A$ is given by 
\begin{equation*}
    K_{A}(x,y)=\int\limits_{\widehat{G}}\textnormal{\textbf{Tr}}(\pi(y^{-1}x)\sigma(x,\pi))d\pi=k_{x}(y^{-1}x).
\end{equation*}
{{In this case, we use  the notation
\begin{equation*}
    A:=\textnormal{Op}(\sigma),
\end{equation*}     
to indicate that $A$ is the pseudo-differential operator associated with symbol  $\sigma.$}}
\end{definition}
 
\subsection{The symbol of a continuous linear operator} The global symbol of  a continuous linear operator on a graded Lie group was computed in Theorem 3.2 of \cite{RuzhanskyDelgadoCardona2019}. We present this formula in Theorem 
\ref{SymbolPseudo}. Let $\mathcal{R}$ be a positive Rockland operator on a graded Lie group. Then $\mathcal{R}$  and $\pi(\mathcal{R}):=d\pi(\mathcal{R})$ (the infinitesimal representation of $\mathcal{R}$) are symmetric and  densely defined operators on $C^\infty_0 (G)$ and $H^\infty_\pi\subset H_\pi.$ We will denote by $\mathcal{R}$  and $\pi(\mathcal{R}):=d\pi(\mathcal{R})$ their self-adjoint extensions to $L^2(G)$ and $H_\pi$ respectively (see Proposition 4.1.5 and Corollary 4.1.16 of \cite[page 178]{FischerRuzhanskyBook2015}). 

\begin{remark}\label{symbolremark}
Let $\mathcal{R}$ be a positive  Rockland operator of  homogeneous degree $\nu$ on a graded Lie group $G.$ Every operator $  \pi(\mathcal{R}) $ has discrete spectrum (see ter Elst and Robinson \cite{TElst+Robinson}) admitting, by the spectral theorem, a basis  contained in its domain. In this case, $H_{\pi}^{\infty}\subset\textnormal{Dom}(\pi(\mathcal{R}))\subset H_{\pi},$ but in view of Proposition 4.1.5 and Corollary 4.1.16 of \cite[page 178]{FischerRuzhanskyBook2015}, every $ \pi(\mathcal{R})$ is densely defined and symmetric on $H^\infty_\pi,$ and this fact allows us to define the (restricted) domain of $\pi(\mathcal{R}),$ as \begin{equation}\label{RestrictedDomain}
    \textnormal{Dom}_{\textnormal{rest}}(\pi(\mathcal{R}))=H_{\pi}^{\infty}.   
\end{equation}Next, when we mention the domain of $\pi(\mathcal{R})$ we are referring to the restricted domain in  \eqref{RestrictedDomain}. This fact will be important, because, via the spectral theorem we can construct a basis for $H_{\pi},$ consisting of vectors in $\textnormal{Dom}_{\textnormal{rest}}(\pi(\mathcal{R}))=H_{\pi}^{\infty},$ where the operator $\pi(\mathcal{R})$ is diagonal. So, if  $B_\pi=\{e_{\pi,k}\}_{k=1}^\infty\subset H_{\pi}^{\infty},$ is  an orthonormal basis such that  $ \pi(\mathcal{R})e_{\pi,k}=\nu_{kk}(\pi)e_{\pi,k},\,k\in \mathbb{N}, \,\,\,\pi \in \widehat{G}, $  the function $x\mapsto \pi(x)e_{\pi,k},$ is smooth and the family of functions
\begin{equation}\label{piij}
   \pi_{ij}:G\rightarrow \mathbb{C},\,\, \pi(x)_{ij}:=( \pi(x)e_{\pi,i},e_{\pi,j} )_{H_\pi},\,\,x\in G,
\end{equation} are smooth functions on $G.$  Consequently, for every continuous linear operator $A:C^\infty(G)\rightarrow C^\infty(G)$ we have \begin{equation*}\{\pi_{ij}\}_{i,j=1}^\infty\subset \textnormal{Dom}(A)=C^\infty(G),\end{equation*} for every $ \pi\in \widehat{G}. $
\end{remark} 

In view of Remark \ref{symbolremark}  we have the following theorem where we present the formula of the  global symbol in terms of its corresponding pseudo-differential operator in the graded setting (see Theorem 3.2 of \cite{RuzhanskyDelgadoCardona2019}).

\begin{theorem}\label{SymbolPseudo}
Let  $\mathcal{R}$ be a positive  Rockland operator of  homogeneous degree $\nu$ on a graded Lie group $G.$ For every $\pi\in \widehat{G},$ let  $B_\pi=\{e_{\pi,k}\}_{k=1}^\infty\subset H_{\pi}^{\infty},$ be  a basis where the operator $\pi(\mathcal{R})$ is diagonal, i.e., 
\begin{equation*}
    \pi(\mathcal{R})e_{\pi,j}=\nu_{jj}(\pi)e_{\pi,j},\,j\in \mathbb{N}_0, \,\,\,\pi \in \widehat{G}.
\end{equation*} For every $x\in G,$ and $\pi\in \widehat{G},$ let us consider the functions $\pi(\cdot)_{ij}\in C^\infty(G)$ in \eqref{piij} induced by the coefficients of the matrix representation
 of $\pi(x)$ in the basis $B_\pi.$ If $A:C^\infty(G)\rightarrow C^\infty(G)$ is a continuous linear operator with symbol 
\begin{equation}\label{definition}
    \sigma:=\{\sigma(x,\pi):H_{\pi}^\infty\rightarrow H_\pi^\infty:\, x\in G,\,\pi\in \widehat{G}\},\, 
\end{equation} such that
\begin{equation}\label{quantization}
    Af(x)=\int\limits_{\widehat{G}}\textnormal{\textbf{Tr}}[\pi(x)\sigma(x,\pi)\widehat{f}(\pi)]d\pi,
\end{equation} for every $f\in \mathscr{S}(G),$ and a.e. $(x,\pi),$ and if   $A\pi(x)$ is the  densely defined operator on $H_{\pi}^\infty,$ via
\begin{equation}\label{thesymbolofA}
  A\pi(x)\equiv  ((A\pi(x)e_{\pi,i},e_{\pi,j}))_{i,j=1}^\infty,\,\,\,(A\pi(x)e_{\pi,i},e_{\pi,j})=:(A\pi_{ij})(x),
\end{equation} then we have
\begin{equation}\label{formulasymbol}
    \sigma(x,\pi)=\pi(x)^*A\pi(x),
\end{equation} for every $x\in G,$ and a.e. $\pi\in \widehat{G}.$ 
\end{theorem}

\subsection{Global H\"ormander classes of pseudo-differential operators}

The main tool in the construction of global H\"ormander classes is the notion of difference operators. 
 Indeed, for every smooth function $q\in C^\infty(G)$ and $\sigma\in L^\infty_{a,b}(G),$ where $a,b\in\mathbb{R},$ the difference operator $\Delta_q$ acts on $\sigma$ according to the formula (see Definition 5.2.1 of \cite{FischerRuzhanskyBook2015}),
 \begin{equation*}
     \Delta_q\sigma(\pi)\equiv [\Delta_q\sigma](\pi):=\mathscr{F}_{G}(qf)(\pi),\,\textnormal{ for  a.e.  }\pi\in\widehat{G},\textnormal{ where }f:=\mathscr{F}_G^{-1}\sigma\,\,.
 \end{equation*}
We will reserve the notation $\Delta^{\alpha}$ for the difference operators defined by the functions $q_{\alpha}(x):=x^\alpha,$ while we denote  by $\tilde{\Delta}^{\alpha}$ the difference operators associated with the functions   $\tilde{q}_{\alpha}(x)=(x^{-1})^\alpha$. In particular, we have the Leibnitz rule,
\begin{equation}\label{differenceespcia;l}
    \Delta^{\alpha}(\sigma\tau)=\sum_{\alpha_1+\alpha_2=\alpha}c_{\alpha_1,\alpha_2}\Delta^{\alpha_1}(\sigma)\Delta^{\alpha_2}(\tau),\,\,\,\sigma,\tau\in L^\infty_{a,b}(\widehat{G}).
\end{equation} 
For our further analysis we will use the following property of the difference operators $\Delta^\alpha,$ (see e.g. \cite[page 20]{FischerFermanian-Kammerer2017}),
\begin{equation}\label{FischerFermanian-Kammerer2017}
    \Delta^{\alpha}(\sigma_{r\cdot})(\pi)=r^{ [\alpha] }(\Delta^{\alpha}\sigma)(r\cdot \pi),\,\,\,r>0\,\,\,\pi\in\widehat{G},
\end{equation}where we have denoted \begin{equation}\label{eeeeeeeee}
    \sigma_{r\cdot}:=\{\sigma(r\cdot \pi):\pi\in\widehat{G}\},\,\,r\cdot \pi(x):=\pi(D_r(x)),\,\,x\in G. 
\end{equation}

In terms of difference operators, the global H\"ormander classes introduced in \cite{FischerRuzhanskyBook2015} can be introduced as follows.
 Let $0\leq \delta,\rho\leq 1,$ and let $\mathcal{R}$ be a positive Rockland operator of homogeneous degree $\nu>0.$ If $m\in \mathbb{R},$ we say that the symbol $\sigma\in L^\infty_{a,b}(\widehat{G}), $ where $a,b\in\mathbb{R},$ belongs to the $(\rho,\delta)$-H\"ormander class of order $m,$ $S^m_{\rho,\delta}(G\times \widehat{G}),$ if for all $\gamma\in \mathbb{R},$ the following conditions
\begin{equation}\label{seminorm}
   p_{\alpha,\beta,\gamma,m}(\sigma)= \operatornamewithlimits{ess\, sup}_{(x,\pi)\in G\times \widehat{G}}\Vert \pi(1+\mathcal{R})^{\frac{\rho [\alpha] -\delta [\beta] -m-\gamma}{\nu}}[X_{x}^\beta \Delta^{\alpha}\sigma(x,\pi)] \pi(1+\mathcal{R})^{\frac{\gamma}{\nu}}\Vert_{\textnormal{op}}<\infty,
\end{equation}
hold true for all $\alpha$ and $\beta$ in $\mathbb{N}_0^n.$ The resulting class $S^m_{\rho,\delta}(G\times \widehat{G}),$ does not depend on the choice of the Rockland operator $\mathcal{R}.$ In particular (see Theorem 5.5.20 of \cite{FischerRuzhanskyBook2015}), the following facts are equivalent: 
\begin{itemize}
    \item[(A)] $\forall \alpha,\beta\in \mathbb{N}_{0}^n, \forall\gamma\in \mathbb{R}, $   $p_{\alpha,\beta,\gamma,m}(\sigma)<\infty;$
    
    \item[(B)]  $\forall \alpha,\beta\in \mathbb{N}_{0}^n, $   $p_{\alpha,\beta,0,m}(\sigma)<\infty;$
    
    \item[(C)]  $\forall \alpha,\beta\in \mathbb{N}_{0}^n, $   $p_{\alpha,\beta,m+\delta [\beta] -\rho [\alpha] ,m}(\sigma)<\infty;$
    \item[(D)] $\sigma\in {S}^{m}_{\rho,\delta}(G\times \widehat{G});$
\end{itemize}and in view of Theorem 13.16 of \cite{CR20}, the conditions (A), (B), (C) and (D) are equivalent to the following one:

\begin{itemize}
    \item[(E)]  $\forall \alpha,\beta\in \mathbb{N}_{0}^n, \exists \gamma_0\in \mathbb{R}, $   $p_{\alpha,\beta,\gamma_0,m}(\sigma)<\infty.$
\end{itemize}

 We will denote,
\begin{equation}
     \Vert \sigma\Vert_{k\,,S^{m}_{\rho,\delta}}:= \max_{ [\alpha] + [\beta] \leq k}\{  p_{\alpha,\beta,0,m}(\sigma)\}.
\end{equation}
\begin{remark}
In the abelian case $G=\mathbb{R}^n,$ endowed with its natural structure of abelian  group, and $\mathcal{R}=-\Delta_{x},$ $x\in \mathbb{R}^n,$ with $\Delta_{x}=\sum_{j=1}^{n}\partial_{x_i}^{2}$ being the usual Laplace operator on $\mathbb{R}^n,$ the classes defined via \eqref{seminorm}, agree with the well known H\"ormander classes on $\mathbb{R}^n$ (see e.g. H\"ormander \cite[Vol. 3]{HormanderBook34}). In this case the difference operators are the partial derivatives on $\mathbb{R}^n,$ (see Remark 5.2.13 and Example 5.2.6 of \cite{FischerRuzhanskyBook2015}).  
\end{remark}
For an arbitrary graded Lie group, the H\"ormander classes $S^m_{\rho,\delta}(G\times \widehat{G}),$ $m\in\mathbb{R},$ provide a symbolic calculus closed under compositions, adjoints, and existence of parametrices. The following theorem summarises the composition and the adjoint rules for global operators  as well as the Calder\'on-Vaillancourth theorem. The notion of ellipticity as developed in \cite{FischerRuzhanskyBook2015} and the  existence of parametrices will be presented in Section \ref{LEH}.
\begin{theorem}\label{calculus} Let $0\leqslant \delta<\rho\leqslant 1,$ and let us denote $\Psi^{m}_{\rho,\delta}:=\textnormal{Op}({S}^{m}_{\rho,\delta}(G\times \widehat{G})),$ for every $m\in \mathbb{R}.$ Then,
\begin{itemize}
    \item [(i)] The mapping $A\mapsto A^{*}:\Psi^{m}_{\rho,\delta}\rightarrow \Psi^{m}_{\rho,\delta}$ is a continuous linear mapping between Fr\'echet spaces and  the  symbol of $A^*,$ $\sigma_{A^*}(x,\pi)\equiv \widehat{A^{*}}(x,\pi) $ satisfies the asymptotic expansion,
 \begin{equation*}
    \widehat{A^{*}}(x,\pi)\sim \sum_{|\alpha|= 0}^\infty\Delta^{\alpha}X_x^\alpha (\widehat{A}(x,\pi)^{*}).
 \end{equation*} This means that, for every $N\in \mathbb{N},$ and all $\ell\in \mathbb{N},$
\begin{equation*}
   \Small{ \Delta^{\alpha_\ell}X_x^\beta\left(\widehat{A^{*}}(x,\pi)-\sum_{|\alpha|\leqslant N}\Delta^\alpha X_x^\alpha (\widehat{A}(x,\pi)^{*}) \right)\in {S}^{m-(\rho-\delta)(N+1)-\rho\ell+\delta[\beta]}_{\rho,\delta}(G\times\widehat{G}) },
\end{equation*} where $[\alpha_\ell]=\ell.$
\item [(ii)] The mapping $(A_1,A_2)\mapsto A_1\circ A_2: \Psi^{m_1}_{\rho,\delta}\times \Psi^{m_2}_{\rho,\delta}\rightarrow \Psi^{m_3}_{\rho,\delta}$ is a continuous bilinear mapping between Fr\'echet spaces, and the symbol of $A=A_{1}\circ A_2,$ satisfies the asymptotic expansion,
\begin{equation*}
    \sigma_A(x,\pi)\sim \sum_{|\alpha|= 0}^\infty(\Delta^\alpha\widehat{A}_{1}(x,\pi))(X_x^\alpha \widehat{A}_2(x,\pi)),
\end{equation*}this means that, for every $N\in \mathbb{N},$ and all $\ell \in\mathbb{N},$
\begin{align*}
    &\Delta^{\alpha_\ell}X_x^\beta\left(\sigma_A(x,\pi)-\sum_{|\alpha|\leqslant N}  (\Delta^{\alpha}\widehat{A}_{1}(x,\pi))(X_x^\alpha \widehat{A}_2(x,\pi))  \right)\\
    &\hspace{2cm}\in {S}^{m_1+m_2-(\rho-\delta)(N+1)-\rho\ell+\delta[\beta]}_{\rho,\delta}(G\times \widehat{G}),
\end{align*}for every  $\alpha_\ell \in \mathbb{N}_0^n$ with $[\alpha_\ell]=\ell.$
\item [(iii)] For  $0\leqslant \delta\leq  \rho\leqslant    1,$  $\delta\neq 1,$ let us consider a continuous linear operator $A:C^\infty(G)\rightarrow\mathscr{D}'(G)$ with symbol  $\sigma\in {S}^{0}_{\rho,\delta}(G\times \widehat{G})$. Then $A$ extends to a bounded operator from $L^2(G)$ to  $L^2(G).$ 
\end{itemize}

\end{theorem}

\subsection{Ellipticity and construction of  parametrices}\label{LEH}

The ellipticity notion of the H\"ormander classes  in the graded setting \cite{FischerRuzhanskyBook2015} is defined by the invertibility of the symbol in high frequencies.  To introduce this notion, for $\Lambda\in \mathbb{R}^+_0$, and for the spectral measures $(E(\lambda))_{\lambda\geq 0}$ and $(E_\pi(\lambda))_{\lambda\geq 0}$ of the positive operators $\mathcal{R}$ and   $\pi(\mathcal{R}),$ respectively, $\pi \in \widehat{G},$ consider the closed subspace
\begin{equation}
    H_{\pi,\Lambda}^{\infty}:=E_{\pi}(\Lambda,\infty)H_{\pi}^\infty.
\end{equation}Now we record Definition 5.8.1 of \cite{FischerRuzhanskyBook2015} as follows.
\begin{definition}
Let $\sigma\in S^{m}_{\rho,\delta}(G\times \widehat{G}).$ The symbol $\sigma$ is elliptic with respect to $\mathcal{R}$ of elliptic order $m_0,$ if there is $\Lambda\in \mathbb{R}^+_0,$ such that for any $\gamma\in \mathbb{R},$ $x\in G,$ and a.e. $\pi\in \widehat{G},$ we have
\begin{equation}
    \forall u\in H_{\pi,\Lambda}^{\infty},\,\,\, \Vert\pi(1+\mathcal{R})^{\frac{\gamma}{\nu}} \sigma(x,\pi)u\Vert_{H_\pi}\geq C_\gamma\Vert \pi(1+\mathcal{R})^{\frac{\gamma+m_0}{\nu}} u\Vert_{H_\pi}.
\end{equation}We also say that the symbol, (and the corresponding operator $\textnormal{Op}(\sigma)$) is $(\mathcal{R},\Lambda, m_0)$-elliptic.
\end{definition}

Now, we will present a technical result about the existence of parametrices for  elliptic  operators. We denote by $ {S}^{-\infty}(G\times \widehat{G}):=\bigcap_{m\in \mathbb{R}}{S}^{m}_{\rho,\delta}(G\times \widehat{G})$ and by $ {\Psi}^{-\infty}(G\times \widehat{G}):=\bigcap_{m\in \mathbb{R}}{\Psi}^{m}_{\rho,\delta}(G\times \widehat{G}),$ the classes of smoothing symbols and smoothing operators, respectively.
\begin{proposition}\label{IesTParametrix} Let $m\in \mathbb{R},$ and let $0\leqslant \delta<\rho\leqslant 1.$  Let  $A=\textnormal{Op}(\sigma)(x,\pi)\in {\Psi}^{m}_{\rho,\delta}(G\times \widehat{G})$ be $(\mathcal{R},\Lambda, m_0)$-elliptic.  Then, there exists $B\in {\Psi}^{-m_0}_{\rho,\delta}(G\times \widehat{G}),$ such that $AB-I,BA-I\in {\Psi}^{-\infty}(G\times \widehat{G}). $ Moreover, the symbol $\tau:=\textnormal{Op}^{-1}(B)$ of $B$ satisfies the following asymptotic expansion
\begin{equation}\label{AE}
    \tau(x,\pi)\sim \sum_{N=0}^\infty\tau_{N}(x,\pi),\,\,\,(x,\pi)\in G\times \widehat{G},
\end{equation}where $\tau_{N}\in {S}^{-m_0-(\rho-\delta)N}_{\rho,\delta}(G\times \widehat{G})$ obeys to the inductive  formula
\begin{equation}\label{conditionelip}
    \tau_{N}(x,\pi)=-\sigma(x,\pi)^{-1}\left(\sum_{k=0}^{N-1}\sum_{|\gamma|=N-k}(\Delta^\gamma \sigma(x,\pi))(X_x^\gamma\tau_{k}(x,\pi))\right),\,\,N\geqslant 1,
\end{equation}with $ \tau_{0}(x,\pi):=E_{\pi}(\Lambda,\infty)\sigma(x,\pi)^{-1}:H_{\pi,\Lambda}^{\infty}\rightarrow H_{\pi}^{\infty}$ being the inverse of $\sigma(x,\pi)$ on $H_{\pi,\Lambda}^{\infty},$ as defined in Lemma 5.8.3 of \cite{FischerRuzhanskyBook2015}.  
\end{proposition} 
 The following is  a criterion of global  ellipticity for all the parameters $\Lambda \in \mathbb{R}^+_0,$ see Corollary 5.8.4 of \cite{FischerRuzhanskyBook2015}.
\begin{proposition} Let $\sigma\in S^{m}_{\rho,\delta}(G\times \widehat{G}),$ and let us assume, that for a.e. $\pi\in \widehat{G},$ and all $x\in G,$ 
$$\sigma(x,\pi):H_{\pi}^{\infty}\rightarrow H_{\pi}^{\infty}$$ is one-to-one and consider its inverse
$$\sigma(x,\pi)^{-1}: H_{\pi}^{\infty}\rightarrow H_{\pi}^\infty.$$ 
The following assertions are equivalent.
\begin{itemize}
    \item $\forall \gamma \in \mathbb{R},$ the restriction operator 
     $$\pi(1+\mathcal{R})^{\frac{\gamma+m_0}{\nu}} \sigma(x,\pi)^{-1}\pi(1+\mathcal{R})^{-\frac{\gamma}{\nu}} |_{H_{\pi}^\infty}:H_{\pi}^{\infty}\rightarrow H_{\pi}$$ admits a bounded extension to $H_\pi,$ with the operator norm uniformly bounded in $(x,\pi),$ that is,
\begin{align*}
    &\sup_{(x,\pi)\in G\times \widehat{G}}\Vert \pi(1+\mathcal{R})^{\frac{\gamma+m_0}{\nu}} \sigma(x,\pi)^{-1}\pi(1+\mathcal{R})^{\frac{\gamma}{\nu}}\Vert_{\textnormal{op}}<\infty.
\end{align*} 
\item  $\sigma$  is $(\mathcal{R},\Lambda, m_0)$-elliptic, for all $\Lambda\in \mathbb{R}^+_0$.
\end{itemize}

\end{proposition}

\begin{definition} Let  $\sigma\in S^{m}_{\rho,\delta}(G\times \widehat{G}),$ be a symbol such that  for a.e. $\pi\in \widehat{G},$ and all $x\in G,$ 
$$\sigma(x,\pi):H_{\pi}^{\infty}\rightarrow H_{\pi}^{\infty}$$ is one-to-one. If $\sigma$ is $(\mathcal{R},\Lambda, m_0)$-elliptic for all $\Lambda\in \mathbb{R}^+_0,$ (that is $\sigma$ is a $(\mathcal{R},\Lambda, 0)$-elliptic symbol)  we say that $\sigma$ is a  globally elliptic symbol  or elliptic for short.   

\end{definition}
\begin{example}For all $\gamma_1,\gamma_2\in \mathbb{R},$ the operator $(1+\mathcal{R})^{\frac{\gamma_1+i\gamma_2}{\nu}}$ has a global symbol defined by the $\widehat{G}$-field of operators
\begin{equation}
  \sigma_{(1+\mathcal{R})^{\frac{\gamma_1+i\gamma_2}{\nu}}}:=  \{(1+\pi(\mathcal{R}))^{\frac{\gamma_1+i\gamma_2}{\nu}}:\pi\in\widehat{G}\}\in S^{\gamma_1}_{1,0}(G\times \widehat{G}),
\end{equation} which is  globally elliptic 
(see Proposition 5.8.2 of \cite{FischerRuzhanskyBook2015}).

\end{example}
\begin{remark}
The $L^p$-mapping properties for the classes $S^{m}_{\rho,\delta}(G\times \widehat{G})$ have been studied in \cite{RuzhanskyDelgadoCardona2019} with the case of $G=\mathbb{R}^n$ proved in \cite{Fefferman1973}.
\end{remark}

\section{Parameter ellipticity on an analytic curve}\label{ParameterGraded} 
In this section we start the construction of our global functional calculus. Firstly, we require the notion of ellipticity with respect to a parameter $\Lambda\in \mathbb{R}.$
We will use
\begin{equation}
    \mathcal{M}:=(1+\mathcal{R})^{\frac{1}{\nu}},
\end{equation} for the Riesz potential of order one associated to the positive Rockland operator $\mathcal{R}$ of homogeneous order $\nu>0.$ For instance, for any $s\in \mathbb{R},$ the field of vectors
\begin{equation}
   \{\pi( \mathcal{M})^s:=(1+\pi(\mathcal{R}))^{\frac{s}{\nu}}:\pi\in \widehat{G}\},
\end{equation}is the global symbol of the operator $\mathcal{M}^s=\int\limits_0^{\infty}(1+\mathcal{\lambda})^{\frac{s}{\nu}}dE(\lambda),$
where $\{E(\lambda)\}_{\lambda>0}$ denotes the spectral measure of $\mathcal{R}.$ The sequence $\{\nu_{jj}(\pi)\}_{j\in \mathbb{N}_0}$ denotes the spectrum of the operator $\pi(\mathcal{R}),$  and by the spectral mapping theorem the sequence $\{(1+\nu_{jj}(\pi))^{\frac{1}{\nu}}\}_{j\in \mathbb{N}_0}$ determines the spectrum of the operator $\pi(\mathcal{M}),$ $\pi \in \widehat{G}.$

To develop the pseudo-differential functional calculus we need a more broad notion of ellipticity. By following the  approach in \cite{RuzhanskyWirth2014}, we present in our pseudo-differential context the notion of parameter  ellipticity. 
\begin{definition} Let $m>0,$ and let $0\leqslant \delta<\rho\leqslant 1.$
Let $\Lambda=\{\gamma(t):t\in I \}$  be an analytic curve in the complex plane $\mathbb{C}$, where $I=[a,b],$ $-\infty<a\leqslant b<\infty,$ $I=[a,\infty),$ $I=(-\infty,b]$ or $I=(-\infty,\infty).$ If $I$ is a finite interval we assume that $\Lambda$ is a closed curve. For  simplicity, if $I$ is an infinite interval we assume that $\Lambda$ is homotopy equivalent to the line $\Lambda_{i\mathbb{R}}:=\{iy:-\infty<y<\infty\}.$  Let  $a=a(x,\pi)\in {S}^{m}_{\rho,\delta}(G\times \widehat{G}).$ Assume also that  
$$
{\boxed{
 R_{\lambda}(x,\pi)^{-1}:=a(x,\pi)-\lambda \textnormal{ is invertible for every } (x,\pi)\in G\times\widehat{G},$ and all $\lambda \in  \Lambda.}} 
$$
We say that $a$ is parameter  elliptic  with respect to $\Lambda,$ if
\begin{equation*}
    \sup_{\lambda\in \Lambda}\sup_{(x,\pi)\in G\times \widehat{G}}\Vert (|\lambda|^{\frac{1}{m}}+\pi(\mathcal{M}))^{m}R_{\lambda}(x,\pi)\Vert_{\textnormal{op}}<\infty.
\end{equation*}
\end{definition}
\begin{remark}
Observe that for $a=b=0,$ $I=\{0\},$ and for the trivial curve $\gamma(t)=0,$ that $a$ is parameter  elliptic  with respect to $\Lambda=\{0\},$ is equivalent to say that $a$ is   elliptic. The following theorem describes the resolvent operator $R_{\lambda}(x,\pi)$ of a parameter  elliptic  symbol $a.$
\end{remark}

\begin{theorem}\label{lambdalambdita} Let $m>0,$ and let $0\leqslant \delta<\rho\leqslant 1.$ If $a$ is parameter  elliptic  with respect to $\Lambda,$ the following estimate 
\begin{equation*}
   \sup_{\lambda\in \Lambda}\sup_{(x,\pi)\in G\times \widehat{G}}\Vert (|\lambda|^{\frac{1}{m}}+\pi(\mathcal{M}))^{m(k+1)}\pi(\mathcal{M})^{\rho[\alpha]-\delta[\beta]}\partial_{\lambda}^kX_x^\beta\Delta^{\alpha}R_{\lambda}(x,\pi)\Vert_{\textnormal{op}}<\infty,
\end{equation*}holds true for all $\alpha,\beta\in \mathbb{N}_0^n$ and $k\in \mathbb{N}_0.$

\end{theorem}
\begin{proof}
    We will split the proof in the cases $|\lambda|\leqslant 1,$  and $|\lambda|> 1,$ where $\lambda\in \Lambda.$ It is possible however that one of these two cases could be trivial in the sense that $\Lambda_{1}:=\{\lambda\in \Lambda:|\lambda|\leqslant 1\}$ or $\Lambda_{1}^{c}:=\{\lambda\in \Lambda:|\lambda|> 1\}$ could be empty sets. In such a case the proof is self-contained in the situation that we will consider where we assume that $\Lambda_{1}$ and $\Lambda_{1}^c$ are not trivial sets.   For $|\lambda|\leqslant 1,$ observe that
\begin{align*}
    &\Vert (|\lambda|^{\frac{1}{m}}+\pi(\mathcal{M}))^{m(k+1)}\pi(\mathcal{M})^{\rho[\alpha]-\delta[\beta]}\partial_{\lambda}^kX_x^\beta\Delta^{\alpha}R_{\lambda}(x,\pi)\Vert_{\textnormal{op}}\\
    &=\Vert (|\lambda|^{\frac{1}{m}}+\pi(\mathcal{M}))^{m(k+1)}\pi(\mathcal{M})^{-m(k+1)}\pi(\mathcal{M})^{m(k+1)+\rho[\alpha]-\delta[\beta]}\partial_{\lambda}^kX_x^\beta\Delta^{\alpha}R_{\lambda}(x,\pi)\Vert_{\textnormal{op}}\\
    &\leqslant \Vert (|\lambda|^{\frac{1}{m}}+\pi(\mathcal{M}))^{m(k+1)}\pi(\mathcal{M})^{-m(k+1)}\Vert_{\textnormal{op}}\\
    &\hspace{5cm}\times\Vert\pi(\mathcal{M})^{m(k+1)+\rho[\alpha]-\delta[\beta]}\partial_{\lambda}^kX_x^\beta\Delta^{\alpha}R_{\lambda}(x,\pi)\Vert_{\textnormal{op}}.
\end{align*} We have
\begin{align*}
    &\Vert (|\lambda|^{\frac{1}{m}}+\pi(\mathcal{M}))^{m(k+1)}\pi(\mathcal{M})^{-m(k+1)}\Vert_{\textnormal{op}}\\
    &= \Vert (|\lambda|^{\frac{1}{m}}\pi(\mathcal{M})^{-1}+I)^{m(k+1)}\Vert_{\textnormal{op}}\leqslant \Vert |\lambda|^{\frac{1}{m}}\pi(\mathcal{M})^{-1}+I\Vert^{m(k+1)}_{\textnormal{op}}\\
    &\leqslant \sup_{|\lambda|\in [0,1]}\sup_{1\leqslant j\leqslant \infty}(|\lambda|^{\frac{1}{m}}(1+\nu_{jj}(\pi))^{-\frac{1}{\nu}}+1))^{k(m+1)}\\
    &=O(1).
\end{align*} On the other hand, we can prove that
\begin{align*}
   \Vert\pi(\mathcal{M})^{m(k+1)+\rho[\alpha]-\delta[\beta]}\partial_{\lambda}^kX_x^\beta\Delta^{\alpha}R_{\lambda}(x,\pi)\Vert_{\textnormal{op}}=O(1).
\end{align*} For $k=1,$ $\partial_{\lambda}R_{\lambda}(x,\pi)= R_{\lambda}(x,\pi)^{2}.$ This can be deduced from the Leibniz rule, indeed,
\begin{align*}
 0=\partial_{\lambda}(R_{\lambda}(x,\pi)(a(x,\pi)-\lambda))=(\partial_{\lambda}R_{\lambda}(x,\pi)) (a(x,\pi)-\lambda)+ R_{\lambda}(x,\pi)(-1) 
\end{align*}implies that
    \begin{align*}
 -\partial_{\lambda}(R_{\lambda}(x,\pi))(a(x,\pi)-\lambda)= R_{\lambda}(x,\pi)(-1). 
\end{align*} Because $(a(x,\pi)-\lambda)=R_{\lambda}(x,\pi)^{-1}$ the identity for the first derivative of $R_\lambda,$ $\partial_{\lambda}R_{\lambda}$  follows. So, from the chain rule we obtain that the term of higher order expanding the derivative   $ \partial_{\lambda}^kR_{\lambda} $ is $ R_{\lambda}^{k+1}.$ Because $R_{\lambda}\in S^{-m}_{\rho,\delta}(G\times \widehat{G}),$ the pseudo-differential calculus implies that $R_{\lambda}^{k+1}\in S^{-m(k+1)}_{\rho,\delta}(G\times \widehat{G}).$ This fact, and the compactness of $\Lambda_1\subset \mathbb{C},$ provide us the uniform estimate 
    \begin{equation*}
   \sup_{\lambda\in \Lambda_1}\sup_{(x,\pi)\in G\times \widehat{G}}\Vert\pi(\mathcal{M})^{m(k+1)+\rho[\alpha]-\delta[\beta]}\partial_{\lambda}^kX_x^\beta\Delta^{\alpha}R_{\lambda}(x,\pi)\Vert_{\textnormal{op}}<\infty.
\end{equation*}Now, we will analyse the situation for $\lambda\in \Lambda_1^c.$ We will use induction over $k$ in order to prove that
\begin{equation*}
   \sup_{\lambda\in \Lambda_1^c}\sup_{(x,\pi)\in G\times \widehat{G}}\Vert (|\lambda|^{\frac{1}{m}}+\pi(\mathcal{M}))^{m(k+1)}\pi(\mathcal{M})^{\rho[\alpha]-\delta[\beta]}\partial_{\lambda}^kX_x^\beta\Delta^{\alpha}R_{\lambda}(x,\pi)\Vert_{\textnormal{op}}<\infty.
\end{equation*}
For $k=0$ notice that
\begin{align*}
     &\Vert (|\lambda|^{\frac{1}{m}}+\pi(\mathcal{M}))^{m(k+1)}\pi(\mathcal{M})^{\rho[\alpha]-\delta[\beta]}\partial_{\lambda}^kX_x^\beta\Delta^{\alpha}R_{\lambda}(x,\pi)\Vert_{\textnormal{op}}\\
     &=\Vert (|\lambda|^{\frac{1}{m}}+\pi(\mathcal{M}))^{m}\pi(\mathcal{M})^{\rho[\alpha]-\delta[\beta]}X_x^\beta\Delta^{\alpha}(a(x,\pi)-\lambda)^{-1}\Vert_{\textnormal{op}},
\end{align*}and denoting $\theta=\frac{1}{|\lambda|},$ $\omega=\frac{\lambda}{|\lambda|},$ we have
 \begin{align*}
     &\Vert (|\lambda|^{\frac{1}{m}}+\pi(\mathcal{M}))^{m(k+1)}\pi(\mathcal{M})^{\rho[\alpha]-\delta[\beta]}\partial_{\lambda}^kX_x^\beta\Delta^{\alpha}R_{\lambda}(x,\pi)\Vert_{\textnormal{op}}\\
     &=\Vert (|\lambda|^{\frac{1}{m}}+\pi(\mathcal{M}))^{m}|\lambda|^{-1}\pi(\mathcal{M})^{\rho[\alpha]-\delta[\beta]}X_x^\beta\Delta^{\alpha}(\theta\times  a(x,\pi)-\omega)^{-1}\Vert_{\textnormal{op}}\\
     &=\Vert (1+|\lambda|^{-\frac{1}{m}}\pi(\mathcal{M}))^{m}\pi(\mathcal{M})^{\rho[\alpha]-\delta[\beta]}X_x^\beta\Delta^{\alpha}(\theta\times  a(x,\pi)-\omega)^{-1}\Vert_{\textnormal{op}}\\
     &=\Vert (1+\theta^{\frac{1}{m}}\pi(\mathcal{M}))^{m}\pi(\mathcal{M})^{\rho[\alpha]-\delta[\beta]}X_x^\beta\Delta^{\alpha}(\theta\times  a(x,\pi)-\omega)^{-1}\Vert_{\textnormal{op}}\\
     &=\Vert (1+\theta^{\frac{1}{m}}\pi(\mathcal{M}))^{m}\pi(\mathcal{M})^{-m}\pi(\mathcal{M})^{m+\rho[\alpha]-\delta[\beta]}X_x^\beta\Delta^{\alpha}(\theta\times  a(x,\pi)-\omega)^{-1}\Vert_{\textnormal{op}}\\
     &\leqslant \Vert (1+\theta^{\frac{1}{m}}\pi(\mathcal{M}))^{m}\pi(\mathcal{M})^{-m}\Vert_{\textnormal{op}}\Vert\pi(\mathcal{M})^{m+\rho[\alpha]-\delta[\beta]}X_x^\beta\Delta^{\alpha}(\theta\times  a(x,\pi)-\omega)^{-1}\Vert_{\textnormal{op}}.
\end{align*}   Because $ (1+\theta^{\frac{1}{m}}\pi(\mathcal{M}))^{m}\pi(\mathcal{M})^{-m} \in S^{0}_{\rho,\delta}(G\times \widehat{G}) ,$ we have that the operator norm $ \Vert (1+\theta^{\frac{1}{m}}\pi(\mathcal{M}))^{m}\pi(\mathcal{M})^{-m}\Vert_{\textnormal{op}}$ is uniformly bounded in $\theta\in [0,1].$ The same argument can be applied to the operator norm $$ \Vert\pi(\mathcal{M})^{m+\rho[\alpha]-\delta[\beta]}X_x^\beta\Delta^{\alpha}(\theta\times  a(x,\pi)-\omega)^{-1}\Vert_{\textnormal{op}}, $$
    by using that $(\theta\times  a(x,\pi)-\omega)^{-1}\in S^{-m}_{\rho,\delta}(G\times \widehat{G}), $ with $\theta\in [0,1]$ and $\omega$ being an element of the complex circle. The case  $k\geqslant 1$ for $\lambda\in \Lambda_1^c$ can be proved in an analogous way.
\end{proof} Combining Proposition \ref{IesTParametrix} and Theorem \ref{lambdalambdita} we obtain the following corollaries.
\begin{corollary}\label{parameterparametrix}
Let $m>0,$ and let $0\leqslant \delta<\rho\leqslant 1.$ Let  $a$ be a parameter  elliptic  symbol with respect to $\Lambda.$ Then  there exists a parameter-dependent parametrix of $A-\lambda I,$ with symbol $a^{-\#}(x,\pi,\lambda)$ satisfying the estimates
\begin{equation*}
   \sup_{\lambda\in \Lambda}\sup_{(x,\pi)\in G\times \widehat{G}}\Vert (|\lambda|^{\frac{1}{m}}+\pi(\mathcal{M}))^{m(k+1)}\pi(\mathcal{M})^{\rho[\alpha]-\delta[\beta]}\partial_{\lambda}^kX_x^\beta\Delta^{\alpha}a^{-\#}(x,\pi,\lambda)\Vert_{\textnormal{op}}<\infty,
\end{equation*}for all $\alpha,\beta\in \mathbb{N}_0^n$ and $k\in \mathbb{N}_0.$
\end{corollary}
\begin{corollary}\label{resolv}
Let $m>0,$ and let $a\in S^{m}_{\rho,\delta}(G\times \widehat{G}) $ where  $0\leqslant \delta<\rho\leqslant 1.$ Let us assume that $\Lambda$ is a subset of the $L^2$-resolvent set of $A,$ $\textnormal{Resolv}(A):=\mathbb{C}\setminus \textnormal{Spec}(A).$ Then $A-\lambda I$ is invertible on $\mathscr{D}'(G)$ and the symbol of the resolvent operator $\mathcal{R}_{\lambda}:=(A-\lambda I)^{-1},$ $\widehat{\mathcal{R}}_{\lambda}(x,\pi)$ belongs to $S^{-m}_{\rho,\delta}(G\times \widehat{G}).$ 
\end{corollary}

\section{Global complex functional calculus on graded Lie groups}\label{SFC}
In this section we develop the global functional calculus for pseudo-differential operators. The calculus will be applied to obtaining a pseudo-differential G\r{a}rding ineqiality and for studying the Dixmier trace of pseudo-differential operators.  

\subsection{Functions of symbols vs functions of operators}\label{S8}
Let $a\in S^{m}_{\rho,\delta}(G\times \widehat{G})$ be a parameter  elliptic  symbol  of order $m>0$ with respect to the sector $\Lambda\subset\mathbb{C}.$ For $A=\textnormal{Op}(a),$ let us define the operator $F(A)$  by the (Dunford-Riesz) complex functional calculus
\begin{equation}\label{F(A)}
    F(A)=-\frac{1}{2\pi i}\oint\limits_{\partial \Lambda_\varepsilon}F(z)(A-zI)^{-1}dz,
\end{equation}where
\begin{itemize}
    \item[(CI)] $\Lambda_{\varepsilon}:=\Lambda\cup \{z:|z|\leqslant \varepsilon\},$ $\varepsilon>0,$ and $\Gamma=\partial \Lambda_\varepsilon\subset\textnormal{Resolv}(A)$ is a positively oriented curve in the complex plane $\mathbb{C}$.
    \item[(CII)] $F$ is an holomorphic function in $\mathbb{C}\setminus \Lambda_{\varepsilon},$ and continuous on its closure. 
    \item[(CIII)] We will assume  decay of $F$ along $\partial \Lambda_\varepsilon$ in order that the operator \eqref{F(A)} will be densely defined on $C^\infty(G)$ in the strong sense of the topology on $L^2(G).$
\end{itemize} Now, we will compute the symbols for operators defined by this complex functional calculus.
\begin{lemma}\label{LemmaFC}
Let $a\in S^{m}_{\rho,\delta}(G\times \widehat{G})$ be a parameter  elliptic  symbol  of order $m>0$ with respect to the sector $\Lambda\subset\mathbb{C}.$ Let $F(A):C^\infty(G)\rightarrow \mathscr{D}'(G)$ be the operator defined by the analytical functional calculus as in \eqref{F(A)}. Under the assumptions $\textnormal{(CI)}$, $\textnormal{(CII)}$, and $\textnormal{(CIII)}$, the symbol of $F(A),$ $\sigma_{F(A)}(x,\pi)$ is given by
\begin{equation*}
    \sigma_{F(A)}(x,\pi)=-\frac{1}{2\pi i}\oint\limits_{\partial \Lambda_\varepsilon}F(z)\widehat{\mathcal{R}}_z(x,\pi)dz
\end{equation*}where $\mathcal{R}_z=(A-zI)^{-1}$ denotes the resolvent of $A,$ and $\widehat{\mathcal{R}}_z(x,\pi)\in S^{-m}_{\rho,\delta}(G\times \widehat{G}) $ is its symbol.
\end{lemma}
\begin{proof}
 From Corollary \ref{resolv}, we have that  $\widehat{\mathcal{R}}_z(x,\pi)\in S^{-m}_{\rho,\delta}(G\times \widehat{G}) .$ Now, observe that  \begin{align*}
  \sigma_{F(A)}(x,\pi)=\pi(x)^*F(A)\pi(x)=-\frac{1}{2\pi i}\oint\limits_{\partial \Lambda_\varepsilon}F(z)\pi(x)^*(A-zI)^{-1}\pi(x)dz.  \end{align*} We finish the proof by observing that $\widehat{\mathcal{R}}_z(x,\pi)=\pi(x)^*(A-zI)^{-1}\pi(x),$ for every $z\in \textnormal{Resolv}(A).$
\end{proof}
Assumption (CIII) will be clarified in the following theorem where we show that the pseudo-differential calculus is stable under the action of the complex functional calculus.
\begin{theorem}\label{DunforRiesz}
Let $m>0,$ and let $0\leqslant \delta<\rho\leqslant 1.$ Let  $a\in S^{m}_{\rho,\delta}(G\times \widehat{G})$ be a parameter  elliptic  symbol with respect to $\Lambda.$ Let us assume that $F$ satisfies the  estimate $|F(\lambda)|\leqslant C|\lambda|^s$ uniformly in $\lambda,$ for some $s<0.$  Then  the symbol of $F(A),$  $\sigma_{F(A)}\in S^{ms}_{\rho,\delta}(G\times \widehat{G}) $ admits an asymptotic expansion of the form
\begin{equation}\label{asymcomplex}
    \sigma_{F(A)}(x,\pi)\sim 
     \sum_{N=0}^\infty\sigma_{{B}_{N}}(x,\pi),\,\,\,(x,\pi)\in G\times \widehat{G},
\end{equation}where $\sigma_{{B}_{N}}(x,\pi)\in {S}^{ms-(\rho-\delta)N}_{\rho,\delta}(G\times \widehat{G})$ and 
\begin{equation*}
    \sigma_{{B}_{0}}(x,\pi)=-\frac{1}{2\pi i}\oint\limits_{\partial \Lambda_\varepsilon}F(z)(a(x,\pi)-z)^{-1}dz\in {S}^{ms}_{\rho,\delta}(G\times \widehat{G}).
\end{equation*}Moreover, 
\begin{equation*}
     \sigma_{F(A)}(x,\pi)\equiv -\frac{1}{2\pi i}\oint\limits_{\partial \Lambda_\varepsilon}F(z)a^{-\#}(x,\pi,\lambda)dz \textnormal{  mod  } {S}^{-\infty}(G\times \widehat{G}),
\end{equation*}where $a^{-\#}(x,\pi,\lambda)$ is the symbol of the parametrix to $A-\lambda I,$   in Corollary \ref{parameterparametrix}.
\end{theorem}
\begin{proof}
    First, we need to prove that the condition $|F(\lambda)|\leqslant C|\lambda|^s$ uniformly in $\lambda,$ for some $s<0,$ is enough in order to guarantee that \begin{equation*}
    \sigma_{{B}_{0}}(x,\pi):=-\frac{1}{2\pi i}\oint\limits_{\partial \Lambda_\varepsilon}F(z)(a(x,\pi)-z)^{-1}dz,
\end{equation*} is a well defined operator-symbol.
From Theorem \ref{lambdalambdita} we deduce that $(a(x,\pi)-z)^{-1}$ satisfies the estimate
\begin{equation*}
   \Vert (|z|^{\frac{1}{m}}+\pi(\mathcal{M}))^{m(k+1)}\pi(\mathcal{M})^{\rho[\alpha]-\delta[\beta]}\partial_{z}^kX_x^\beta\Delta^{\alpha}(a(x,\pi)-z)^{-1}\Vert_{\textnormal{op}}<\infty.
\end{equation*}
Observe that 
\begin{align*}
    &\Vert(a(x,\pi)-z)^{-1}\Vert_{\textnormal{op}}\\
    & =\Vert (|z|^{\frac{1}{m}}+\pi(\mathcal{M}))^{-m}(|z|^{\frac{1}{m}}+\pi(\mathcal{M}))^{m}(a(x,\pi)-z)^{-1}\Vert_{\textnormal{op}} \\
    &\lesssim  \sup_{1\leqslant j\leqslant \infty}(|z|^{\frac{1}{m}}+(1+\nu_{jj}(\pi))^{\frac{1}{\nu}})^{-m}\\&\leqslant |z|^{-1},
\end{align*} and the condition $s<0$ implies
\begin{align*}
    \left|\frac{1}{2\pi i}\oint\limits_{\partial \Lambda_\varepsilon}F(z)(a(x,\pi)-z)^{-1}dz\right|\lesssim \oint\limits_{\partial \Lambda_\varepsilon}|z|^{-1+s}|dz|<\infty,
\end{align*}uniformly in $(x,\pi)\in G\times \widehat{G}.$ In order to check that $\sigma_{B_0}\in {S}^{ms}_{\rho,\delta}(G\times \widehat{G})$ let us analyse the cases $-1<s<0$ and $s\leqslant -1$ separately. So, let us analyse first the situation of $-1<s<0.$ We observe that
\begin{align*}
   &\Vert \pi(\mathcal{M})^{-ms+\rho[\alpha]-\delta[\beta]}X_x^\beta\Delta^{\alpha}\sigma_{B_0}(x,\pi)\Vert_{\textnormal{op}}\\
   &\leqslant \frac{C}{2\pi }\oint\limits_{\partial \Lambda_\varepsilon} |z|^{s}\Vert      \pi(\mathcal{M})^{-ms+\rho[\alpha]-\delta[\beta]}X_x^\beta\Delta^{\alpha}(a(x,\pi)-z)^{-1}\Vert_{\textnormal{op}} |dz|.
\end{align*}Now, we will estimate the operator norm inside of the integral. Indeed, the identity
\begin{align*}
    &\Vert      \pi(\mathcal{M})^{-ms+\rho[\alpha]-\delta[\beta]}X_x^\beta\Delta^{\alpha}(a(x,\pi)-z)^{-1}\Vert_{\textnormal{op}}=\\
    &\Vert (|z|^{\frac{1}{m}}+\pi(\mathcal{M}))^{-m}(|z|^{\frac{1}{m}}+\pi(\mathcal{M}))^{m}\pi(\mathcal{M})^{-ms+\rho[\alpha]-\delta[\beta]}X_x^\beta\Delta^{\alpha}(a(x,\pi)-z)^{-1}\Vert_{\textnormal{op}}
\end{align*}implies that
\begin{align*}
    &\Vert      \pi(\mathcal{M})^{-ms+\rho[\alpha]-\delta[\beta]}X_x^\beta\Delta^{\alpha}(a(x,\pi)-z)^{-1}\Vert_{\textnormal{op}} \lesssim  \Vert (|z|^{\frac{1}{m}}+\pi(\mathcal{M}))^{-m}\pi(\mathcal{M})^{-ms}\Vert_{\textnormal{op}}
\end{align*}where we have used that
\begin{align*}
\sup_{z\in \partial \Lambda_\varepsilon}  \sup_{(x,\pi)} \Vert (|z|^{\frac{1}{m}}+\pi(\mathcal{M}))^{m}\pi(\mathcal{M})^{\rho[\alpha]-\delta[\beta]}X_x^\beta\Delta^{\alpha}(a(x,\pi)-z)^{-1}\Vert_{\textnormal{op}} <\infty.
\end{align*}
Consequently, by using that  $s<0,$ we deduce
\begin{align*}
   & \frac{C}{2\pi }\oint\limits_{\partial \Lambda_\varepsilon} |z|^{s}\Vert      \pi(\mathcal{M})^{ms+\rho[\alpha]-\delta[\beta]}X_x^\beta\Delta^{\alpha}(a(x,\pi)-z)^{-1}\Vert_{\textnormal{op}} |dz|\\
    &\lesssim \frac{C}{2\pi }\oint\limits_{\partial \Lambda_\varepsilon} |z|^{s}\Vert (|z|^{\frac{1}{m}}+\pi(\mathcal{M}))^{-m} \pi(\mathcal{M})^{-ms}\Vert_{\textnormal{op}} |dz|\\
    &= \frac{C}{2\pi }\oint\limits_{\partial \Lambda_\varepsilon} |z|^{s}\sup_{1\leqslant j\leqslant \infty} (|z|^{\frac{1}{m}}+(1+\nu_{jj}(\pi))^{\frac{1}{\nu}})^{-m}(1+\nu_{jj}(\pi))^{-\frac{ms}{\nu}}|dz|.
\end{align*}
In order to study the convergence of the last contour integral we only need to check the convergence of $\int_{1}^{\infty}r^s(r^{\frac{1}{m}}+\varkappa)^{-m}\varkappa^{-ms}dr,$ where $\varkappa>1$ in a parameter. The change of variable $r=\varkappa^{m}t$ implies that
\begin{align*}
   \int\limits_{1}^{\infty}r^s(r^{\frac{1}{m}}+\varkappa)^{-m}\varkappa^{-ms}dr&=\int\limits_{\varkappa^{-m}}^{\infty}\varkappa^{ms}t^s(\varkappa t^{\frac{1}{m}}+\varkappa)^{-m}\varkappa^{-ms}\varkappa^mdt=\int\limits_{\varkappa^{-m}}^{\infty}t^s(t^{\frac{1}{m}}+1)^{-m}dt\\
   &\lesssim \int\limits_{\varkappa^{-m}}^{1}t^sdt+\int\limits_{1}^{\infty}t^{-1+s}<\infty.
\end{align*}Indeed, for $t\rightarrow\infty,$ $t^s(t^{\frac{1}{m}}+1)^{-m}\lesssim t^{-1+s}$ and we conclude the estimate because $\int\limits_{1}^{\infty} t^{-1+s'}dt<\infty,$ for all $s'<0.$ On the other hand, the condition $-1<s<0$ implies that
\begin{align*}
 \int\limits_{\varkappa^{-m}}^{1}t^sdt=\frac{1}{1+s}-\frac{\varkappa^{-m(1+s)}}{1+s}=   O(1).
\end{align*} In the case where $s\leqslant -1,$ we can find an analytic function $\tilde{G}(z)$ such that it is a holomorphic function in $\mathbb{C}\setminus \Lambda_{\varepsilon},$ and continuous on its closure and additionally satisfying that $F(\lambda)=\tilde{G}(\lambda)^{1+[-s]}.$\footnote{ $[-s]$ denotes  the integer part of $-s.$} In this case,  $\tilde{G}(A)$ defined by the complex functional calculus 
\begin{equation}\label{G(A)}
    \tilde{G}(A)=-\frac{1}{2\pi i}\oint\limits_{\partial \Lambda_\varepsilon}\tilde{G}(z)(A-zI)^{-1}dz,
\end{equation}
has symbol belonging to ${S}^{\frac{sm}{1+[-s]}}_{\rho,\delta}(G\times \widehat{G})$ because $\tilde{G}$ satisfies the estimate $|G(\lambda)|\leqslant C|\lambda|^{\frac{s}{1+[-s]}},$ with $-1<\frac{s}{1+[-s]}<0.$ 
By observing that
\begin{align*}
    \sigma_{F(A)}(x,\pi)&=-\frac{1}{2\pi i}\oint\limits_{\partial \Lambda_\varepsilon}F(z)\widehat{\mathcal{R}}_z(x,\pi)dz=-\frac{1}{2\pi i}\oint\limits_{\partial \Lambda_\varepsilon}\tilde{G}(z)^{1+[-s]}\widehat{\mathcal{R}}_z(x,\pi)dz\\
    &=\sigma_{\tilde{G}(A)^{1+[-s]}}(x,\pi),
\end{align*}and computing the symbol $\sigma_{\tilde{G}(A)^{1+[-s]}}(x,\pi)$ by iterating $1+[-s]$-times  the asymptotic formula for the composition in the pseudo-differential calculus, we can see that the term with higher order in such expansion is $\sigma_{\tilde{G}(A)}(x,\pi)^{1+[-s]}\in {S}^{ms}_{\rho,\delta}(G\times \widehat{G}).$ Consequently we have proved that $\sigma_{F(A)}(x,\pi)\in {S}^{ms}_{\rho,\delta}(G\times \widehat{G}).$
This completes the proof for the first part of the theorem.
For the second part of the proof, let us denote by $a^{-\#}(x,\pi,\lambda)$  the symbol of the parametrix to $A-\lambda I,$   in Corollary \ref{parameterparametrix}. Let $P_{\lambda}=\textnormal{Op}(a^{-\#}(\cdot,\cdot,\lambda)).$ Because $\lambda\in \textnormal{Resolv}(A)$ for $\lambda\in \partial \Lambda_\varepsilon,$ $(A-\lambda)^{-1}-P_{\lambda}$ is an smoothing operator. Consequently, from Lemma \ref{LemmaFC} we deduce that
\begin{align*}
   & \sigma_{F(A)}(x,\pi)\\
   &=-\frac{1}{2\pi i}\oint\limits_{\partial \Lambda_\varepsilon}F(z)\widehat{\mathcal{R}}_z(x,\pi)dz\\
    &=-\frac{1}{2\pi i}\oint\limits_{\partial \Lambda_\varepsilon}F(z)a^{-\#}(x,\pi,z)dz-\frac{1}{2\pi i}\oint\limits_{\partial \Lambda_\varepsilon}F(z)(\widehat{\mathcal{R}}_z(x,\pi)-a^{-\#}(x,\pi,z))dz\\
    &\equiv -\frac{1}{2\pi i}\oint\limits_{\partial \Lambda_\varepsilon}F(z)a^{-\#}(x,\pi,z)dz  \textnormal{  mod  } {S}^{-\infty}(G\times \widehat{G}).
\end{align*}The asymptotic expansion \eqref{asymcomplex} comes from the construction of the parametrix in the pseudo-differential calculus (see Proposition \ref{IesTParametrix}).
\end{proof}

\subsection{G\r{a}rding inequality}\label{Gardinggraded} In the setting of graded groups the  global version of the G\r{a}rding inequality remains valid. Indeed, in  \cite{FischerRuzhanskyLowerBounds} it was proved under the assumption that the symbol of the operator commutes with the symbol of $\mathcal{R}$. In this section we remove this commutativity condition making use of the complex functional calculus for the H\"ormander classes allowing the G\r{a}rding inequality to be valid in a more general context.  To do so, we need some preliminary results. 
\begin{proposition}
Let $0\leqslant \delta<\rho\leqslant 1.$  Let  $a\in S^{m}_{\rho,\delta}(G\times \widehat{G})$ be an  elliptic  symbol where $m\geqslant 0$ and let us assume that $a(x,\pi)$ is positive definite for a.e. $(x,\pi)\in G\times \widehat{G}$. Then $a$ is parameter-elliptic with respect to $\mathbb{R}_{-}:=\{z=x+i0:x<0\}\subset\mathbb{C}.$ Furthermore, for any number $s\in \mathbb{C},$ 
\begin{equation*}
    \widehat{B}(x,\pi)\equiv a(x,\pi)^s:=\exp(s\log(a(x,\pi))),\,\,(x,\pi)\in G\times \widehat{G},
\end{equation*}defines a symbol $\widehat{B}(x,\pi)\in S^{m\times\textnormal{Re}(s)}_{\rho,\delta}(G\times \widehat{G}).$
\end{proposition}
\begin{proof}  Any operator $a(x,\pi),$ $\pi\in \widehat{G},$ is  normal, so that for every $\lambda\in \mathbb{R}_{-}$ we have
    \begin{align*}
      &   \Vert (|\lambda|^{\frac{1}{m}}+\pi(\mathcal{M}))^m (a(x,\pi)-\lambda)^{-1}\Vert_{\textnormal{op}}\\
         &\asymp \Vert (|\lambda|^{\frac{1}{m}}+\pi(\mathcal{M}))^m (\pi(\mathcal{M})^m-\lambda)^{-1}(\pi(\mathcal{M})^m-\lambda)(a(x,\pi)-\lambda)^{-1}\Vert_{\textnormal{op}} \\
         &\lesssim \Vert (|\lambda|^{\frac{1}{m}}+\pi(\mathcal{M}))^m (\pi(\mathcal{M})^m-\lambda)^{-1}\Vert_{\textnormal{op}} \Vert (\pi(\mathcal{M})^m-\lambda)(a(x,\pi)-\lambda)^{-1}\Vert_{\textnormal{op}}\\
         &\lesssim \Vert (|\lambda|^{\frac{1}{m}}+\pi(\mathcal{M}))^m (\pi(\mathcal{M})^m-\lambda)^{-1}\Vert_{\textnormal{op}}.
    \end{align*}
{{Let us note that the condition $m\geq 0,$ implies that $\Vert \pi(\mathcal{M}))^{-m} \Vert_{\textnormal{op}}\leq 1.$ So, if $|\lambda|\leq 1/2,$ then   $|\lambda|\Vert \pi(\mathcal{M}))^{-m} \Vert_{\textnormal{op}}\leq 1/2,$ which implies that for $|\lambda|\leq 1/2,$ $1-\lambda \pi(\mathcal{M}))^{-m}$ is invertible on $H_\pi,$ its inverse is a bounded operator, and from
the first von-Neumann identity  
 $$ \Vert (1-\lambda \pi(\mathcal{M}))^{-m})^{-1}\Vert_{\textnormal{op}}\leq  ( 1-|\lambda| \Vert\pi(\mathcal{M})\Vert_{\textnormal{op}})^{-m})^{-1}\leq 2. $$ } } 
    
    Now, fixing again $\lambda\in \mathbb{R}_{-}$ observe that from the compactness of $[0,1/2]$ we deduce that
    \begin{align*}
        \sup_{0\leqslant \lambda\leqslant 1/2}\Vert (|\lambda|^{\frac{1}{m}}+\pi(\mathcal{M}))^m (\pi(\mathcal{M})^m-\lambda)^{-1}\Vert_{\textnormal{op}}&\asymp \sup_{0\leqslant \lambda\leqslant 1/2}\Vert \pi(\mathcal{M})^m (\pi(\mathcal{M})^m-\lambda)^{-1}\Vert_{\textnormal{op}} \\
        &\asymp \sup_{0\leqslant \lambda\leqslant 1/2}\Vert \pi(\mathcal{M})^m (\pi(\mathcal{M})^m-\lambda)^{-1}\Vert_{\textnormal{op}} \\
         &\asymp \sup_{0\leqslant \lambda\leqslant 1/2}\Vert  (I-\lambda\pi(\mathcal{M})^{-m})^{-1}\Vert_{\textnormal{op}} \\
          &\lesssim 1.
    \end{align*} On the other hand,
    \begin{align*}
    &  \sup_{\lambda\geqslant 1/2}\Vert (|\lambda|^{\frac{1}{m}}+\pi(\mathcal{M}))^m (\pi(\mathcal{M})^m-\lambda)^{-1}\Vert_{\textnormal{op}}  \\
      &=\sup_{\lambda\geqslant 1/2}\Vert (|\lambda|^{\frac{1}{m}}\pi(\mathcal{M})^{-1}+I)^m (I-\pi(\mathcal{M})^{-m}\lambda)^{-1}\Vert_{\textnormal{op}}\\
      &=\sup_{\lambda\geqslant 1/2}\Vert (\pi(\mathcal{M})^{-1}+|\lambda|^{-\frac{1}{m}}I)^m |\lambda|(I-\pi(\mathcal{M})^{-m}\lambda)^{-1}\Vert_{\textnormal{op}}\\
      &\lesssim \sup_{\lambda\geqslant 1/2}\Vert \pi(\mathcal{M})^{-m} |\lambda|(-\lambda)^{-1}\pi(\mathcal{M})^{m}\Vert_{\textnormal{op}}\\
      &=1.
    \end{align*}So, we have proved that $a$ is parameter-elliptic with respect to $\mathbb{R}_{-}.$ To prove that $\widehat{B}(x,\pi)\in S^{m\times\textnormal{Re}(s)}_{\rho,\delta}(G\times \widehat{G}),$ we can observe that for $\textnormal{Re}(s)<0,$ we can apply Theorem \ref{DunforRiesz}. If  $\textnormal{Re}(s)\geqslant 0,$ we can find $k\in \mathbb{N}$ such that $\textnormal{Re}(s)-k<0$ and consequently from the spectral calculus of operators we deduce that $a(x,\pi)^{\textnormal{Re}(s)-k}\in S^{m\times(\textnormal{Re}(s)-k)}_{\rho,\delta}(G\times \widehat{G}).$ So, from the calculus we conclude that   $$a(x,\pi)^{s}=a(x,\pi)^{s-k}a(x,\pi)^{k}\in S^{m\times\textnormal{Re}(s)}_{\rho,\delta}(G\times \widehat{G}).$$ Thus the proof is complete.
\end{proof}
\begin{corollary}\label{1/2}
Let $0\leqslant \delta< \rho\leqslant 1.$  Let  $a\in S^{m}_{\rho,\delta}(G\times \widehat{G}),$  be an  elliptic  symbol  where $m\geqslant 0$ and let us assume that $a$ is positive definite. Then 
$\widehat{B}(x,\pi)\equiv a(x,\pi)^\frac{1}{2}:=\exp(\frac{1}{2}\log(a(x,\pi)))\in S^{\frac{m}{2}}_{\rho,\delta}(G\times \widehat{G}).$
\end{corollary}
Now, let us assume that  for some positive real number $C_0>0,$ the symbol
\begin{equation*}
    A(x,\pi):=\frac{1}{2}(a(x,\pi)+a(x,\pi)^{*}),\,(x,\pi)\in G\times \widehat{G},\,\,a\in S^{m}_{\rho,\delta}(G\times \widehat{G}), 
\end{equation*}is elliptic and satisfies 
\begin{align*}
  A(x,\pi)\geqslant C_0\pi(\mathcal{M})^{m},
\end{align*} in the operator sense, that is, $ A(x,\pi)-C_0\pi(\mathcal{M})^{m} \geqslant 0.$ Then, for $C_1\in(0, {C_0})$ we have that
\begin{align*}
 A(x,\pi)-C_{1}  \pi(\mathcal{M})^{m}\geqslant \left({C_0}-C_1\right) \pi(\mathcal{M})^{m}>0.
\end{align*}If  $0\leqslant \delta<\rho\leqslant 1,$   from Corollary \ref{1/2}, we have that
\begin{align*}
    q(x,\pi):=(A(x,\pi)-C_{1}  \pi(\mathcal{M})^{m})^{\frac{1}{2}}\in  S^{\frac{m}{2}}_{\rho,\delta}(G\times \widehat{G}).
\end{align*}From the symbolic calculus we obtain
\begin{align*}
  q(x,\pi)q(x,\pi)^*= A(x,\pi)-C_{1}  \pi(\mathcal{M})^{m}+r(x,\pi),\,\,   r(x,\pi)\in  S^{m-(\rho-\delta)}_{\rho,\delta}(G\times \widehat{G}).
\end{align*}Now, let us assume that $u\in C^\infty(G).$ Then we have
\begin{align*}
    \textnormal{Re}(a(x,D)u,u)&=\frac{1}{2}((a(x,D)+\textnormal{op}(a^*))u,u)=(A(x,D)u,u)\\
    &=C_{1}(\mathcal{M}_{m}u,u)+(q(x,D)q(x,D)^{*}u,u)+(r(x,D)u,u)\\
    &=C_{1}(\mathcal{M}_{m}u,u)+(q(x,D)^{*}u,q(x,D)^*u)-(r(x,D)u,u)\\
    &\geqslant C_{1}\Vert u\Vert_{{L}^{2}_{\frac{m}{2}}(G)}-(r(x,D)u,u)\\
     &= C_{1}\Vert u\Vert_{{L}^{2}_{\frac{m}{2}}(G)}-(\mathcal{M}^{-\frac{m-(\rho-\delta)}{2}}r(x,D)u,\mathcal{M}^{\frac{m-(\rho-\delta)}{2}}u).
\end{align*}Observe that
\begin{align*}
    (\mathcal{M}^{-\frac{m-(\rho-\delta)}{2}}r(x,D)u,\mathcal{M}^{\frac{m-(\rho-\delta)}{2}}u)&\leqslant \Vert \mathcal{M}^{-\frac{m-(\rho-\delta)}{2}}r(x,D)u \Vert_{L^2(G)}\Vert u\Vert_{L^{2  }_{\frac{m-(\rho-\delta)}{2}}(G)}\\
    &= \Vert r(x,D)u \Vert_{L^{2  }_{-\frac{m-(\rho-\delta)}{2}}(G)}\Vert u\Vert_{L^{2  }_{\frac{m-(\rho-\delta)}{2}}(G)}\\
     &\leqslant C_1\Vert u \Vert_{L^{2  }_{\frac{m-(\rho-\delta)}{2}}(G)}\Vert u\Vert_{L^{2  }_{\frac{m-(\rho-\delta)}{2}}(G)},
\end{align*}where in the last line we have used the Sobolev boundedness of $r(x,D)$ from $L^{2  }_{\frac{m-(\rho-\delta)}{2}}(G)$ into $L^{2  }_{-\frac{m-(\rho-\delta)}{2}}(G).$ Consequently, we deduce the lower bound
\begin{align*}
    \textnormal{Re}(a(x,D)u,u) \geqslant C_{1}\Vert u\Vert_{{L}^{2}_{\frac{m}{2}}(G)}-C\Vert u\Vert_{L^{2  }_{\frac{m-(\rho-\delta)}{2}}(G)}^2.
\end{align*} If we temporarily assume that for every $\varepsilon>0,$ there exists $C_{\varepsilon}>0,$ such that
\begin{equation}\label{lemararo}
    \Vert u\Vert_{{L}^{2}_{\frac{m-(\rho-\delta)}{2}}(G)}^2\leqslant \varepsilon\Vert u\Vert_{{L}^{2}_{\frac{m}{2}}(G)}^2+C_{\varepsilon}\Vert u \Vert_{L^2(G)}^2,
\end{equation} for $0<\varepsilon<C_{1}$ we have
\begin{align*}
    \textnormal{Re}(a(x,D)u,u) \geqslant (C_{1}-\varepsilon)\Vert u\Vert_{{L}^{2}_{\frac{m}{2}}(G)}-C_{\varepsilon}\Vert u\Vert_{L^2(G)}^2.
\end{align*}So, with the exception of the proof of \eqref{lemararo} we have deduced the following estimate which is the main result of this subsection.
\begin{theorem}[G\r{a}rding inequality for graded Lie groups]\label{GardinTheorem}  For $0\leqslant \delta<\rho\leqslant 1,$ let $a(x,D):C^\infty(G)\rightarrow\mathscr{D}'(G)$ be an operator with symbol  $a\in {S}^{m}_{\rho,\delta}( G\times \widehat{G})$, $m\in \mathbb{R}$. Let us assume that for some positive real number $C_0>0,$ the real part of  $a(x,\pi),$ that is
\begin{equation*}
    A(x,\pi):=\frac{1}{2}(a(x,\pi)+a(x,\pi)^{*}),\,(x,\pi)\in G\times \widehat{G},\,\,a\in S^{m}_{\rho,\delta}(G\times \widehat{G}), 
\end{equation*}is elliptic and satisfies
\begin{align*}
  A(x,\pi)\geqslant C_0\pi(\mathcal{M})^{m},\quad \pi \in \widehat{G},
\end{align*}  in the operator sense.  Then, there exist $C_{1},C_{2}>0,$ such that the lower bound
\begin{align}
    \textnormal{Re}(a(x,D)u,u) \geqslant C_1\Vert u\Vert_{{L}^{2}_{\frac{m}{2}}(G)}-C_2\Vert u\Vert_{L^2(G)}^2,
\end{align}holds true for every $u\in C^\infty(G).$
\end{theorem}
In view of the analysis above, for the proof of Theorem \ref{GardinTheorem} we only need to prove \eqref{lemararo}. However we will deduce it from the following more general lemma.
\begin{lemma}Let us assume that $s\geqslant t\geqslant 0$ or that $s,t<0.$ Then, for every $\varepsilon>0,$ there exists $C_\varepsilon>0$ such that 
\begin{equation}\label{lemararo2}
    \Vert u\Vert_{{L}^{2}_{t}(G)}^2\leqslant \varepsilon\Vert u\Vert_{{L}^{2}_{s}(G)}^2+C_{\varepsilon}\Vert u \Vert_{L^2(G)}^2,
\end{equation}holds true for every $u\in C^\infty(G).$ 
\end{lemma}
\begin{proof}
    Let $\varepsilon>0.$ Then, there exists $C_{\varepsilon}>0$ such that
    \begin{equation*}
        (1+\nu_{ii}(\pi))^{2t/\nu}-\varepsilon (1+\nu_{ii}(\pi))^{2s/\nu}\leqslant C_{\varepsilon},
    \end{equation*}uniformly in  $\pi\in \widehat{G}.$ Then \eqref{lemararo2}  follows from the Plancherel theorem. Indeed,
    \begin{align*}
     \Vert u\Vert_{{L}^{2}_{\frac{t}{2}}(G)}^2& =\int\limits_{ \widehat{G}}\sum_{i,j=1}^{\infty} (1+\nu_{ii}(\pi))^{2t/\nu}|\widehat{u}_{ij}(\pi)|^{2}d\pi  \\
     &\leqslant \int\limits_{ \widehat{G}}\sum_{i,j=1}^{\infty} (\varepsilon(1+\nu_{ii}(\pi))^{2s/\nu}+C_\varepsilon)|\widehat{u}_{ij}(\pi)|^{2} d\pi \\
    &= \varepsilon\Vert u\Vert_{{L}^{2}_{s}(G)}^2+C_{\varepsilon}\Vert u \Vert_{L^2(G)}^2,
    \end{align*}completing the proof.
\end{proof}
\begin{corollary}\label{GardinTheorem2}  Let $a(x,D):C^\infty(G)\rightarrow\mathscr{D}'(G)$ be an operator with symbol  $a\in {S}^{m}_{\rho,\delta}( G\times \widehat{G})$, $m\in \mathbb{R}$. Let us assume that for some $C_0>0,$ the symbol $a(x,\pi)$ is elliptic and satisfies
\begin{align*}
  a(x,\pi)\geqslant C_0\pi(\mathcal{M})^{m},
\end{align*}  in the operator sense.  Then, there exist $C_{1},C_{2}>0,$ such that the lower bound
\begin{align}
    \textnormal{Re}(a(x,D)u,u) \geqslant C_1\Vert u\Vert_{{L}^{2}_{\frac{m}{2}}(G)}-C_2\Vert u\Vert_{L^2(G)}^2,
\end{align}holds true for every $u\in C^\infty(G).$
\end{corollary}

\section{Analysis of diffusion problems on graded groups}\label{GST}
The  complex 
pseudo-differential functional calculus here developed for the global H\"ormander classes on graded Lie groups  will be applied to the    well-posedness  of parabolic and hyperbolic problems.
More precisely, if $T>0$ we will study the regularity, existence and uniqueness  for the following Cauchy problem 
\begin{equation}\label{PVI}(\textnormal{IVP}): \begin{cases}\frac{\partial v}{\partial t}=K(t,x,D)v+f ,& \text{ }v\in \mathscr{D}'((0,T)\times G),
\\v(0)=u_0 ,& \text{ } \end{cases}
\end{equation}where the initial data $u_0\in H^s(G),$ $K(t):=K(t,x,D)\in \Psi^{m}_{\rho,\delta}(G\times \widehat{G}),$ $f\in  L^2([0,T],H^s(G)) ,$ $m>0,$ $s\in \mathbb{R},$ and  a suitable positivity condition is imposed on $K.$ Indeed, we consider the operator
\begin{equation*}
    \textnormal{Re}(K(t)):=\frac{1}{2}(K(t)+K(t)^*),\,\,0\leqslant t\leqslant T,
\end{equation*}to be  elliptic, and in order to use the G\r{a}rding inequality (see Theorem \ref{GardinTheorem}),  we will assume that there is a positive real number $C_0>0,$ such that for any $t,$ 
\begin{align*}
  \sigma_{K(t)}(x,\pi)\geqslant C_0\pi(\mathcal{M})^{m},\quad \pi \in \widehat{G},
\end{align*}  in the operator sense.
We recall that for any $s\in \mathbb{R},$ the Sobolev space $H^{s}(G)\equiv L^2_s(G)$ is defined as the completion of $C^{\infty}_0(G)$ with respect to the norm
$$  \Vert u\Vert_{H^s(G)}\equiv \Vert u\Vert_{L^2_s(G)} :=\Vert (1+\mathcal{R})^{\frac{s}{\nu}}u\Vert_{L^2(G)}.  $$

Under such assumptions we will prove the existence and uniqueness of a solution $v\in C^1([0,T],L^2(G))\cap C([0,T],H^{m/2}(G)).$ We will start with the following energy estimate.
\begin{proposition}\label{energyestimate}
Let $K(t)=K(t,x,D)\in \Psi^{m}_{\rho,\delta}(G\times \widehat{G}),$ $0\leqslant \delta<\rho\leqslant  1,$  be a  pseudo-differential operator of order $m>0,$ such that: 
\begin{itemize}
    \item $\textnormal{Re}(K(t))$ is  elliptic  for every  $t\in[0,T]$ with $T>0.$
    \item The symbol $\sigma_{K(t)}(x,\pi)$ of $K(t)$ satisfies the operator inequality $$\sigma_{K(t)}(x,\pi)\geqslant C_0\pi(\mathcal{M})^{m}, \pi \in \widehat{G},$$ for some $C_0>0.$  
\end{itemize}
  If  $v\in C^1([0,T], L^2(G) )  \cap C([0,T],H^{m/2}(G)) ,$ then there exist $C,C'>0,$ such that
\begin{equation}
\Vert v(t) \Vert^2_{L^2(G)}\leqslant   \left(C\Vert v(0) \Vert^2_{L^2(G)}+C'\int\limits_{0}^T \Vert (\partial_{t}-K(\tau))v(\tau) \Vert^2_{L^2(G)}d\tau \right),  
\end{equation}holds for every $0\leqslant t\leqslant T.$ Moreover, we also have the estimate
\begin{equation}\label{Q*}
\Vert v(t) \Vert^2_{L^2(G)}\leqslant   \left(C\Vert v(T) \Vert^2_{L^2(G)}+C'\int\limits_{0}^T \Vert (\partial_{t}-K(\tau)^{*})v(\tau) \Vert^2_{L^2(G)}d\tau \right).  
\end{equation}
\end{proposition}
\begin{proof} 
Let $v\in C^1([0,T], L^2(G) )  \cap C([0,T],H^{m/2}(G)) .$ The hypothesis   $$v\in  C([0,T],H^{\frac{m}{2}}(G))$$ will be used in order to apply the G\r{a}rding inequality. This fact will be useful later because we will use the G\r{a}rding inequality applied to the operator $\textnormal{Re}(-K(t)).$ So, 
$v\in \textnormal{Dom}(\partial_{\tau}-K(\tau))$ for every $0\leqslant \tau\leqslant T.$ In view of the embedding $H^{m/2}\hookrightarrow L^2(G),$ we also have that  $v\in C([0,T], L^2(G) ).$  Let us define $f(\tau):=Q(\tau)v(\tau),$ with $Q(\tau):=(\partial_{\tau}-K(\tau)),$ for every $0\leqslant \tau\leqslant T.$ Observe that

\begin{align*}
   \frac{d}{dt}\Vert v(t) \Vert^2_{L^2(G)}&= \frac{d}{dt}\left(v(t),v(t)\right)_{L^2(G)}\\&=\left(\frac{d v(t)}{dt},v(t)\right)_{L^2(G)}+\left(v(t),\frac{d v(t)}{dt}\right)_{L^2(G)}\\
   &=\left(K(t)v(t)+f(t),v(t)\right)_{L^2(G)}+\left(v(t),K(t)v(t)+f(t)\right)_{L^2(G)}\\
    &=\left((K(t)+K(t)^{*})v(t),v(t)\right)_{L^2(G)}+2\textnormal{Re}(f(t), v(t))_{L^2(G)}\\
     &=2\textnormal{Re}(K(t)v(t),v(t))_{L^2(G)}+2\textnormal{Re}(f(t), v(t))_{L^2(G)}.
\end{align*}Now, from the  G\r{a}rding inequality, 
\begin{align}\label{applying :Garding}
    \textnormal{Re}(-K(t)v(t),v(t)) \geqslant C_1\Vert v(t)\Vert_{H^{\frac{m}{2}}(G)}-C_2\Vert v(t)\Vert_{L^2(G)}^2,
\end{align}and  the parallelogram law, we have
\begin{align*}
 2\textnormal{Re}(f(t), v(t))_{L^2(G)}&\leqslant 2\textnormal{Re}(f(t), v(t))_{L^2(G)}+\Vert f(t)\Vert_{L^2(G)}^2+\Vert v(t)\Vert_{L^2(G)}^2 \\
 &= \Vert f(t)+v(t)\Vert^2\leqslant \Vert f(t)+v(t)\Vert^2+\Vert f(t)-v(t)\Vert^2 \\
&= 2\Vert f(t)\Vert^2_{L^2(G)}+2\Vert v(t)\Vert^2_{L^2(G)},
\end{align*}
and consequently
\begin{align*}
   & \frac{d}{dt}\Vert v(t) \Vert^2_{L^2(G)}\\
   &\leqslant 2\left(C_2\Vert v(t)\Vert_{L^2(G)}^2-C_1\Vert v(t)\Vert_{H^{\frac{m}{2}}(G)}\right)+2\Vert f(t)\Vert^2_{L^2(G)}+2\Vert v(t)\Vert^2_{L^2(G)}.
\end{align*}  So, we have proved that
\begin{align*}
   \frac{d}{dt}\Vert v(t) \Vert^2_{L^2(G)}\lesssim  \Vert f(t)\Vert^2_{L^2(G)}+\Vert v(t)\Vert^2_{L^2(G)}.
\end{align*}By using Gronwall's Lemma we obtain the energy estimate
\begin{equation}
\Vert v(t) \Vert^2_{L^2(G)}\leqslant  \left(C\Vert v(0) \Vert_{L^2(G)}^2+C'\int\limits_{0}^T \Vert f(\tau) \Vert_{L^2(G)}^2d\tau \right),   
\end{equation}for every $0\leqslant t\leqslant T,$ and $T>0.$
To finish the proof, we can change the analysis above with $v(T-\cdot)$ instead of $v(\cdot),$ $f(T-\cdot)$ instead of $f(\cdot)$ and $Q^{*}=-\partial_{t}-K(t)^{*},$ (or equivalently $Q=\partial_{t}-K(t)$ ) instead of $Q^{*}=-\partial_{t}+K(t)^{*}$ (or equivalently $Q=\partial_{t}-K(t)$) using that $\textnormal{Re}(K(T-t)^*)=\textnormal{Re}(K(T-t))$ to deduce that
\begin{align*}
&\Vert v(T-t) \Vert^2_{L^2(G)}\\
&\leqslant \left(C\Vert v(T) \Vert^2_{L^2(G)}+C'\int\limits_{0}^{T} \Vert (-\partial_{t}+K(T-t)^{*})v(T-\tau) \Vert^2_{L^2(G)}d\tau \right)\\
&=   \left(C\Vert v(T) \Vert^2_{L^2(G)}+C'\int\limits_{0}^T \Vert (-\partial_{t}-K(t)^{*})v(s) \Vert^2_{L^2(G)}ds \right).
\end{align*}So, we conclude the proof.
\end{proof}
Now, we will prove the existence and uniqueness for \eqref{PVI} when the initial data are taken in $L^2(G).$
\begin{theorem}\label{Main:Th}
Let $K(t)=K(t,x,D)\in \Psi^{m}_{\rho,\delta}(G\times \widehat{G}),$ $0\leqslant \delta<\rho\leqslant  1,$  be a  pseudo-differential operator of order $m>0,$ such that: 
\begin{itemize}
    \item $\textnormal{Re}(K(t))$ is  elliptic  for every  $t\in[0,T]$ with $T>0.$
    \item The symbol $\sigma_{K(t)}(x,\pi)$ of $K(t)$ satisfies the operator inequality $$\sigma_{K(t)}(x,\pi)\geqslant C_0\pi(\mathcal{M})^{m}, \pi \in \widehat{G},$$ for some $C_0>0.$  
\end{itemize}
Let   $u_0\in L^2(G)$, and let $f\in L^2([0,T],L^2(G)).$ Then there exists a unique $v\in C^1([0,T], L^2(G) )  \cap C([0,T],H^{m/2}(G)) $ solving \eqref{PVI}. Moreover, $v$ satisfies the energy estimate
\begin{equation}
\Vert v(t) \Vert^2_{L^2(G)}\leqslant  \left(C\Vert u_0 \Vert^2_{L^2(G)}+C'\Vert f \Vert^2_{L^2([0,T],L^2(G))} \right),
\end{equation}for every $0\leqslant t\leqslant T.$
\end{theorem}
\begin{proof}
{{We have proved the {\it a priori} estimate in Proposition \ref{energyestimate} for  $$v\in X_{m}:=C^1([0,T], L^2(G) )  \cap C([0,T],H^{m/2}(G)) .$$ Essentially the condition $v\in  C([0,T],H^{m/2}(G)) $ was used when applying the G\r{a}rding inequality in \eqref{applying :Garding}. Now, we will use the Hanh-Banach theorem to deduce that the solution to \eqref{PVI} indeed, lives in $X_{m}.$ }} So, let us denote  $Q=\partial_{t}-K(t),$ and let us introduce the spaces
    \begin{align*}
        E:=\{\phi\in  C^1([0,T], L^2(G) )  \cap C([0,T],H^{m/2}(G)) :\phi(T)=0\},
    \end{align*}
    and $ Q^*E:=\{Q^*\phi:\phi\in E\}.$ Let us define the linear form $\beta\in (Q^*E)'$ by
    \begin{equation*}
    \beta(Q^*\phi):=\int\limits_{0}^T(f(\tau),\phi(\tau))d\tau+\frac{1}{i}(u_0,\phi(0)).
    \end{equation*}
    From \eqref{Q*} we deduce  for every $\phi\in E$ that
    \begin{align*}
    \Vert \phi(t) \Vert^2_{L^2(G)}&\leqslant   \left(C\Vert \phi(T) \Vert^2_{L^2(G)}+C'\int\limits_{0}^T \Vert (\partial_{t}-K(\tau)^{*})\phi(\tau) \Vert^2_{L^2(G)}d\tau \right).
    \end{align*} So, we have
    \begin{align*}
    |\beta(Q^*\phi)|&\leqslant \int\limits_{0}^T\Vert f(\tau)\Vert_{L^2(G)}\Vert \phi(\tau)\Vert_{L^2(G)} d\tau+\Vert u_0\Vert_{{L^2(G)}}\Vert \phi(0)\Vert_{L^2(G)}\\
    &\leqslant \Vert f\Vert_{L^2([0,T],L^2(G))}\Vert \phi\Vert_{L^2([0,T],L^2(G))}+\Vert u_0\Vert_{L^2(G)}\Vert \phi\Vert_{L^2(G)}\\
    &\lesssim_{T} (\Vert f\Vert_{L^2([0,T],L^2(G)}+\Vert u_0\Vert_{L^2(G)})\Vert Q^*\phi(\tau) \Vert_{L^2([0,T],L^2(G))},
    \end{align*}which shows that $\beta$ is a bounded functional on $\mathscr{T}:=Q^*E\cap L^2([0,T],L^2(G)),$ with $\mathscr{T}$ endowed with the topology induced by the  norm of $L^2([0,T],L^2(G)).$ By using the Hahn-Banach extension theorem, we can extends $\beta$ to a bounded functional $\tilde{\beta}$ on $L^2([0,T],L^2(G)),$ and by using the Riesz representation theorem, there exists $v\in (L^2([0,T],L^2(G))'=L^2([0,T],L^2(G)) ,$ such that
    \begin{equation*}
      \tilde{\beta}(\psi)=(v,\psi),\quad \psi\in   L^2([0,T],L^2(G)).
    \end{equation*}In particular, for $\psi=Q^*\phi\in \mathscr{T},$ we have
    \begin{equation*}
      \tilde{\beta}(Q^* \phi)={\beta}(Q^*\phi)=(v,Q^*\phi),
    \end{equation*}  Because, we can identify $C^{\infty}_{0}((0,T),\mathscr{D}(G))$ as a subspace of $E$ 
    \begin{equation*}
       C^{\infty}_{0}((0,T),C^\infty(G))\subset E=\{\phi\in  C^1([0,T], L^2(G) )  \cap C([0,T],H^{m/2}(G)) :\phi(T)=0\},
    \end{equation*} we have the identity
    \begin{align*}
      (f,\phi)=  \int\limits_{0}^T(f(\tau),\phi(\tau))d\tau= \int\limits_{0}^T(f(\tau),\phi(\tau))d\tau+\frac{1}{i}(u_0,\phi(0))=(v,Q^*\phi),
    \end{align*}for every $\phi \in C^{\infty}_{0}((0,T),C^\infty(G)). $ So, this implies that $v\in \textnormal{Dom}(Q^{**}).$ Because $Q^{**}=Q,$ we have that
    \begin{align*}
        (v,Q^*\phi)=(Qv,\phi)=(f,\phi),\,\,\forall\phi \in C^{\infty}_{0}((0,T),C^\infty(G)),
    \end{align*}which implies that $Qv=f.$ A routine argument of integration by parts shows that $v(0)=u_0.$ Now, in order to show the uniqueness of $v,$ let us assume that $u\in  C^1([0,T], L^2(G) )  \cap C([0,T],H^{m/2}(G))$ is a solution of the problem
    \begin{equation*} \begin{cases}\frac{\partial u}{\partial t}=K(t,x,D)u+f ,& \text{ }u\in \mathscr{D}'((0,T)\times G),
\\u(0)=u_0 .& \text{ } \end{cases}
\end{equation*} Then $\omega:=v-u\in C^1([0,T], L^2(G) )  \cap C([0,T],H^{m/2}(G))$ solves the problem
\begin{equation*} \begin{cases}\frac{\partial \omega}{\partial t}=K(t,x,D)\omega,& \text{ }\omega\in \mathscr{D}'((0,T)\times G),
\\\omega(0)=0 ,& \text{ } \end{cases}
\end{equation*}and from Proposition \ref{energyestimate}, $\Vert \omega(t)\Vert_{L^2(G)}=0,$ for all $0\leqslant t\leqslant T,$ and consequently, from the continuity in $t$ of the functions we have that $v(t,x)=u(t,x)$ for all $t\in [0,T]$ and a.e. $x\in G.$
\end{proof}
Now, we extend Theorem \ref{Main:Th} to the case where the initial data are considered in general Sobolev spaces (we consider the additional hypothesis that $K(t)$ commutes with  $\mathcal{R}$).

\begin{theorem}\label{Main:Th:Sob}
Let $K(t)=K(t,x,D)\in \Psi^{m}_{\rho,\delta}(G\times \widehat{G}),$ $0\leqslant \delta<\rho\leqslant  1,$  be a  pseudo-differential operator of order $m>0,$ such that: 
\begin{itemize}
    \item $\textnormal{Re}(K(t))$ is  elliptic  for every  $t\in[0,T]$ with $T>0.$
    \item The symbol $\sigma_{K(t)}(x,\pi)$ of $K(t)$ satisfies the operator inequality $$\sigma_{K(t)}(x,\pi)\geqslant C_0\pi(\mathcal{M})^{m}, \pi \in \widehat{G},$$ for some $C_0>0.$ 
    \item $K(t)$ commutes with the Rockland operator $\mathcal{R}.$
\end{itemize} Let $s\in \mathbb{R},$    $u_0\in H^s(G)$, and let $f\in L^2([0,T],H^s(G)).$ Then there exists a unique $v\in C^1([0,T], H^s(G) )  \cap C([0,T],H^{m/2+s}(G)) $ solving \eqref{PVI}. Moreover, $v$ satisfies the energy estimate
\begin{equation}
\Vert v(t) \Vert^2_{H^s(G)}\leqslant  \left(C\Vert u_0 \Vert^2_{H^s(G)}+C'\Vert f \Vert^2_{L^2([0,T],H^s(G))} \right),
\end{equation}for every $0\leqslant t\leqslant T.$
\end{theorem}

\begin{proof}Observe that if $v$ is a solution for the Cauchy problem
\begin{equation*} \begin{cases}\frac{\partial v}{\partial t}=K(t,x,D)v+f ,& \text{ }v\in \mathscr{D}'((0,T)\times G),
\\v(0)=u_0 ,& \text{ } \end{cases}
\end{equation*}where the initial data $u_0\in H^{s}(G),$ $K(t):=K(t,x,D)\in \Psi^{m}_{\rho,\delta}(G\times \widehat{G}),$ $f\in  L^2([0,T],H^s(G)) ,$ then using the commutativity of $K(t)$ with $\mathcal{R},$ (and hence with its spectral calculus) we have
\begin{align*}
    \frac{\partial }{\partial t}(1+\mathcal{R})^{\frac{s}{\nu}} v &= (1+\mathcal{R})^{\frac{s}{\nu}} \frac{\partial v}{\partial t} =(1+\mathcal{R})^{\frac{s}{\nu}}K(t,x,D)v+(1+\mathcal{R})^{\frac{s}{\nu}}f\\
    &=K(t,x,D)(1+\mathcal{R})^{\frac{s}{\nu}}
v+(1+\mathcal{R})^{\frac{s}{\nu}}f,
\end{align*}showing that $v':= (1+\mathcal{R})^{\frac{s}{\nu}}
v\in L^2(G)$ and $f':=(1+\mathcal{R})^{\frac{s}{\nu}}f\in L^2([0,T],L^2(G))$ satisfy the initial boundary value problem
\begin{equation*} \begin{cases}\frac{\partial v'}{\partial t}=K(t,x,D)v'+f' ,& \text{ }v'\in \mathscr{D}'((0,T)\times G),
\\v'(0)=u'_0:=(1+\mathcal{R})^{\frac{s}{\nu}}u_0\in L^2(G). & \text{ } \end{cases}
\end{equation*}So, by applying Proposition \ref{Main:Th} we get the estimate
\begin{equation}\label{aux:sec:2}
\Vert v'(t) \Vert^2_{L^2(G)}\leqslant  \left(C\Vert u'_0 \Vert^2_{L^2(G)}+C'\Vert f' \Vert^2_{L^2([0,T],L^2(G))} \right),
\end{equation}for every $0\leqslant t\leqslant T.$ Note that by applying Proposition \ref{Main:Th} we have  that $v'\in C([0,T],H^m(G))$ which implies that $v\in C([0,T],H^{m+s}(G)).$   Finally, observe that \eqref{aux:sec:2} is equivalent to the following energy inequality
\begin{equation}\label{aux:sec}
\Vert v(t) \Vert^2_{H^s(G)}\leqslant  \left(C\Vert u_0 \Vert^2_{H^s(G)}+C'\Vert f \Vert^2_{H^s([0,T],L^2(G))} \right),
\end{equation}for every $0\leqslant t\leqslant T,$ proving Theorem \ref{Main:Th:Sob}.
\end{proof}

\noindent{\textbf{Acknowledgement.}} The authors thank David Rottensteiner for discussions.

\bibliographystyle{amsplain}

\end{document}